\documentclass[11pt]{amsart}
\usepackage{amsfonts}
\usepackage{amsmath}
\usepackage{amscd}
\usepackage{amsthm}
\usepackage{enumerate}

\usepackage{graphicx}

\usepackage{color}

\usepackage{hyperref}

\newtheorem{theorem}{Theorem} 
\newtheorem*{theorem*}{Theorem}
\newtheorem{lemma}[theorem]{Lemma}
\newtheorem{definition}[theorem]{Definition}
\newtheorem{proposition}[theorem]{Proposition}

\newtheorem{corollary}[theorem]{Corollary}

\theoremstyle{remark}
\newtheorem{rmk}[theorem]{Remark}

\newcommand{\pp}{\mathbb{P}}
\newcommand{\qq}{\mathbb{Q}}
\newcommand{\rr}{\mathbb{R}}

\newcommand{\nn}{\mathbb{N}}

\newcommand{\eq}{\begin{equation}}
\newcommand{\en}{\end{equation}}

\newcommand{\ev}{\mathbb{E}}

\newcommand{\J} {J}
\newcommand{\JQ} {\widetilde J}
\newcommand{\JM}{\mathcal J}

\newcommand{\Y} {\mathbb Y}

\newcommand{\eps}{\varepsilon}

\newcommand{\GT}{\mathbb{GT}}

\renewcommand{\P}{\mathbb P}

\newcommand{\y}{\mathbf y}

\newcommand{\GG}{\overline{\mathcal{G}^N}}
\newcommand{\GGo}{{\mathcal{G}^N}}

\numberwithin{equation}{section} \numberwithin{theorem}{section}

\title{Multilevel Dyson Brownian motions via Jack polynomials}

\author{Vadim Gorin}
\address{Department of Mathematics, Massachusetts Institute of Technology, Cambridge, MA, USA and Institute for Information Transmission Problems of Russian Academy of Sciences, Moscow, Russia}
\email{vadicgor@gmail.com}
\author{Mykhaylo Shkolnikov}
\address{Department of Statistics, University of California, Berkeley, CA, USA}
\email{mshkolni@gmail.com}

\begin{document}

\begin{abstract}
We introduce multilevel versions of Dyson Brownian motions of arbitrary parameter $\beta>0$,
generalizing the interlacing reflected Brownian motions of Warren for $\beta=2$. Such processes
unify $\beta$ corners processes and Dyson Brownian motions in a single object. Our approach is
based on the approximation by certain multilevel discrete Markov chains of independent interest,
which are defined by means of Jack symmetric polynomials. In particular, this approach allows to
show that the levels in a multilevel Dyson Brownian motion are intertwined (at least for $\beta\ge
1$) and to give the corresponding link explicitly.
\end{abstract}

\maketitle

\tableofcontents

\section{Introduction}

\subsection{Preface}
The Hermite general $\beta>0$ ensemble of rank $N$ is a probability distribution on the set of $N$ tuples
of reals $z_1<z_2<\ldots<z_N$ whose density is proportional to
\begin{equation}
\label{eq_beta_Hermite} \prod_{1\le i<j \le N} (z_j-z_i)^\beta \; \prod_{i=1}^N
\exp\left(-\frac{z_i^2}{2}\right).
\end{equation}
When $\beta=2$, that density describes the joint distribution of the eigenvalues of a random
Hermitian $N\times N$ matrix $M$, whose diagonal entries are i.i.d. real standard normal random
variables, while real and imaginary parts of its entries above the diagonal are i.i.d. normal
random variables of variance $1/2$. The law of such a random matrix is referred to as the Gaussian
Unitary Ensemble (GUE) (see e.g.\ \cite{Meh}, \cite{AGZ}, \cite{For}) and it has attracted much
attention in the mathematical physics literature following the seminal work of Wigner in the 50s.
Similarly, the case of $\beta=1$ describes the joint distribution of eigenvalues of a real
symmetric matrix sampled from the Gaussian Orthogonal Ensemble (GOE) and the case $\beta=4$
corresponds to the Gaussian Symplectic Ensemble (GSE) (see e.g. \cite{Meh}, \cite{AGZ}, \cite{For}
for the detailed definitions).

\medskip

It is convenient to view the realizations of \eqref{eq_beta_Hermite} as a point process on the
real line and there are two well-known ways of adding a second dimension to that picture. The
first one is to consider an $N$-dimensional diffusion known as Dyson Brownian motion (see
\cite[Chapter 9]{Meh}, \cite[Section 4.3]{AGZ} and the references therein) which is the unique
strong solution of the system of stochastic differential equations
\begin{equation}
 \label{eq_Dyson_BM}
\mathrm{d}X_i(t)=\frac{\beta}{2}\,\sum_{j\neq i} \frac{1}{X_i(t)-X_j(t)}\,\mathrm{d}t +
\mathrm{d}W_i(t), \quad i=1,2,\ldots,N
\end{equation}
with $W_1,W_2,\ldots,W_N$ being independent standard Brownian motions. If one solves
\eqref{eq_Dyson_BM} with zero initial condition, then the distribution of the solution at time $1$
is given by \eqref{eq_beta_Hermite}. When $\beta=2$ and one starts with zero initial condition, the
diffusion \eqref{eq_Dyson_BM} has two probabilistic interpretations: it can be either viewed as a
system of $N$ independent standard Brownian motions conditioned never to collide via a suitable
Doob's $h$-transform; or it can be regarded as the evolution of the eigenvalues of a Hermitian
random matrix whose elements evolve as (independent) Brownian motions.

\medskip

An alternative way of adding a second dimension to the ensemble in \eqref{eq_beta_Hermite} involves the
so-called corner processes. For $\beta=2$ take a $N\times N$ GUE matrix $M$ and let
$x^N_1\le x^N_2\le \ldots\le x^N_N$ be its ordered eigenvalues. More generally for every $1\le k\le N$ let
$x^k_1\le x^k_2 \le \ldots\le x^k_k$ be the eigenvalues of the top-left $k\times k$ submatrix (``corner'') of
$M$. It is well-known that the eigenvalues \emph{interlace} in the sense that $x_i^k\le x_i^{k-1}\le x_{i+1}^k$ for
$i=1,\dots,k-1$ (see Figure \ref{Fig_interlace} for a schematic illustration of the eigenvalues).

\begin{figure}[h]
\begin{center}
 {\scalebox{0.8}{\includegraphics{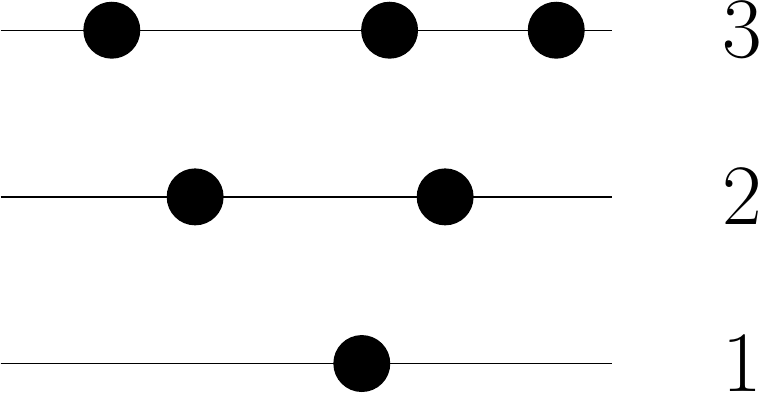}}}
\end{center}
\caption{Interlacing particles arising from eigenvalues of corners of a $3\times 3$ matrix. Row
number $k$ on the picture corresponds to eigenvalues of the $k\times k$ corner.}
\label{Fig_interlace}
\end{figure}

The joint distribution of $x^k_i$, $1\leq i\leq k\leq N$ is known as GUE-corners process (some
authors also use the name GUE-minors process) and its study was initiated in \cite{Bar} and
\cite{JN}. The GUE-corners process is uniquely characterized by two properties: its projection to
the set of particles $x^N_1,x^N_2,\ldots,x^N_N$ is given by \eqref{eq_beta_Hermite} with
$\beta=2$, and the conditional distribution of $x^k_i$, $1\leq i\leq k\leq N-1$ given
$x^N_1,x^N_2,\ldots,x^N_N$ is \emph{uniform} on the polytope defined by the interlacing conditions
above, see \cite{GN}, \cite{Bar}. Due to the combination of the gaussianity and uniformity
embedded into its definition, the GUE--corners process appears as a universal scaling limit for a
number of 2d models of statistical mechanics, see \cite{OR}, \cite{JN}, \cite{GP}, \cite{G-ASM},
\cite{GS}.

\medskip

Similarly, one can construct corners processes for $\beta=1$ and $\beta=4$, see e.g.\ \cite{N}.
Extrapolating the resulting formulas for the joint density of eigenvalues to general values of
$\beta>0$ one arrives at the following definition.

\begin{definition}\label{defbetacorner} The Hermite $\beta$ corners process of variance $t>0$ is the unique
probability distribution on the set of reals $x^k_i$, $1\leq i\leq k\leq N$ subject to the interlacing
conditions $x_i^k\le x_i^{k-1}\le x_{i+1}^k$ whose density is proportional to
\begin{equation}
\label{eq_beta_Hermite_corners_intro}
 \prod_{i<j} (x_j^N-x_i^N)
 \prod_{i=1}^N \exp\left(- \frac{(x_i^N)^2}{2t}\right)
 \prod_{k=1}^{N-1} \prod_{1\le i<j\le k} (x_j^k-x_i^k)^{2-\beta} \prod_{a=1}^k \prod_{b=1}^{k+1} |x^k_a-x^{k+1}_b|^{\beta/2-1}.
\end{equation}
\end{definition}
The fact that the projection of the Hermite $\beta$ corners process of variance $1$ onto level $k$
(that is, on the coordinates $x_1^k,x^k_2,\ldots,x_k^k$) is given by the corresponding Hermite
$\beta$ ensemble of \eqref{eq_beta_Hermite} can be deduced from the Dixon-Anderson integration
formula (see \cite{Dixon}, \cite{And}), which was studied before in the context of \emph{Selberg
integrals} (see \cite{Sel}, \cite[Chapter 4]{For}). One particular case of the Selberg integral is
the evaluation of the normalizing constant for the probability density of the Hermite $\beta$
ensemble of \eqref{eq_beta_Hermite}. We provide more details in this direction in Section
\ref{Section_limit_at_fixed_time}.

\bigskip

The ultimate goal of the present article is to \emph{combine} Dyson Brownian motions and corner
processes in a single picture. In other words, we aim to introduce a relatively simple diffusion
on interlacing particle configurations whose projection on a fixed level is given by Dyson
Brownian motion of \eqref{eq_Dyson_BM}, while its fixed time distributions are given by the
Hermite $\beta$ corners processes of Definition \ref{defbetacorner}.

\medskip

One would think that a natural way to do this (at least for $\beta=1,2,4$) is to consider a $N\times
N$ matrix of suitable Brownian motions and to project it onto the (interlacing) set of eigenvalues of the matrix
and its top--left $k\times k$ corners, thus generalizing the original construction of Dyson.
However, the resulting stochastic process ends up being quite nasty even in the case $\beta=2$. It
is shown in \cite{ANV} that already for $N=3$ (and at least \emph{some} initial conditions) the
projection is \emph{not} a Markov process. When one considers only two adjacent levels (that is, the
projection onto $x_1^N,x^N_2,\ldots,x_N^N; x_1^{N-1},x^{N-1}_2,\ldots,x_{N-1}^{N-1}$), then it can be
proven (see \cite{ANV}) that the projection \emph{is} Markovian, but the corresponding SDE is very complicated.

\medskip

An alternative elegant solution for the case $\beta=2$ was given by Warren \cite{W}. Consider the
process $(Y^k_i:\;1\leq i\leq k\leq N)$ defined through the following inductive procedure: $Y^1_1$
is a standard Brownian motion with zero initial condition; given $Y^1_1$, the processes $Y^2_1$ and
$Y^2_2$ are constructed as independent standard Brownian motions started at zero and
\emph{reflected} on the trajectory of $Y^1_1$ in such a way that $Y^2_1(t)\le Y^1_1(t)\le Y^2_2(t)$
holds for all $t\geq0$. More generally, having constructed the processes on the first $k$ levels
(that is, $Y^m_i$, $1\leq i\leq m\leq k$) one defines $Y^{k+1}_i$ as an independent standard
Brownian motion started at $0$ and reflected on the trajectories of $Y^k_{i-1}$ and $Y^k_{i}$ in
such a way that $Y^{k-1}_{i-1}(t)\le Y^k_i(t)\le Y^{k-1}_i(t)$ remains true for all $t\geq0$ (see
\cite{W} and also \cite{GS} for more details). Warren shows that the projection of the dynamics on
a level $k$ (that is, on $Y^k_1,Y^k_2,\ldots,Y^k_k$) is given by a $k$-dimensional Dyson Brownian
Motion of \eqref{eq_Dyson_BM} with $\beta=2$, and that the fixed time distributions of the process
$(Y^k_i:\;1\leq i\leq k\leq N)$ are given by the Hermite $\beta$ corners processes of Definition
\ref{defbetacorner} with $\beta=2$.

\medskip

Our aim is to construct a generalization of the Warren process for general values of $\beta$. In
other words, we want to answer the question ``What is the general $\beta$ analogue of the
reflected interlacing Brownian Motions of \cite{W}?''.

\subsection{Our results}

Our approach to the construction of the desired general $\beta$ multilevel stochastic process is
based on its discrete space approximation. In \cite{GS} we proved that the reflected interlacing
Brownian motions of \cite{W} can be obtained as a diffusive scaling limit for a class of
stochastic dynamics on \emph{discrete} interlacing particle configurations. The latter dynamics
are constructed from independent random walks by imposing the \emph{local} block/push interactions
between particles to preserve the interlacing conditions. The special cases of such processes
arise naturally in the study of 2d statistical mechanics systems such as random stepped surfaces
and various tilings (cf.\ \cite{BF}, \cite{Nordenstam}, \cite{BG_shuf}, \cite{B-Schur}).

In Section \ref{sec:disc} we introduce a deformation $X^{multi}_{disc}(t)$ of these processes
depending on a real parameter $\theta$ (which is omitted from the notation; $\theta=1$ corresponds
to the previously known case). The resulting discrete space dynamics is an intriguing interacting
particle system with \emph{global} interactions whose state space is given by interlacing particle
configurations with integer coordinates. Computer simulations of this dynamics for $\theta=1/2$ and
$\theta=2$ can be found at \cite{Han}.

We further study the diffusive limit of $X^{multi}_{disc}(s)$ under the rescaling of time
$s=\eps^{-1} t$ and $\eps^{-1/2}$ scaling of space as $\eps\downarrow 0$. Our first result is that
for any fixed $\theta>0$, the rescaled processes are \emph{tight} as $\eps\downarrow 0$, see
Theorem \ref{Theorem_multilevel_tight} for the exact statement. The continuous time, continuous
space processes $Y^{mu}(t)$ defined as \emph{subsequential limits} $\eps\downarrow 0$ of the
family $X^{multi}_{disc}(s)$, are our main heros and we prove a variety of results about them for
different values of $\theta$.

\begin{enumerate}
 \item For any $\theta\ge 2$ we show in Theorem \ref{Theorem_Limit_SDE} that $Y^{mu}(t)$ satisfies a system of SDEs
 \eqref{eq_intDBM_intro}.
 \item For any $\theta\ge 1/2$ and any $1\le k\le N$ we show in Theorem \ref{Theorem_restriction_DBM} that if $Y^{mu}(t)$ is started from
 a
 \emph{$\theta$--Gibbs} initial condition (zero initial condition is a particular case), then
  the $k$--dimensional restriction of $N(N-1)/2$
 dimensional process $Y^{mu}(t)$ to the level $k$ is $2\theta$--Dyson Brownian motion, i.e.\ the
 vector $(Y^{mu}(t)^k_1, Y^{mu}(t)^k_2,\dots, Y^{mu}(t)^k_k)$ solves \eqref{eq_Dyson_BM} with $\beta=2\theta$.
 \item For any $\theta> 0$ we show  that if $Y^{mu}(t)$ is started from zero initial condition,
 then its distribution at time $t$ is the Hermite $2\theta$ corners process of variance $t$, i.e.\
 \eqref{eq_beta_Hermite_corners_intro} with $\beta=2\theta$. In fact, we prove a more general statement, see Theorem
 \ref{Theorem_restriction_DBM} and Corollary \ref{cor_zero_initial}.
 \item For $\theta=1$ using the results of \cite{GS} one shows that $Y^{mu}(t)$ is  the reflected interlacing
Brownian motions of \cite{W}.
\end{enumerate}

The above results are complemented by the following uniqueness theorem for the system of SDEs
\eqref{eq_intDBM_intro}. In particular, it implies that for $\theta\ge 2$ all the subsequential
limits $Y^{mu}(t)$ of $X^{multi}_{disc}(s)$ as $\eps\downarrow 0$, are the same.

\begin{theorem}[Theorem
\ref{theorem_intDBM}]  For any $N\in\nn$ and $\theta>1$, the system of SDEs
\begin{equation}
\label{eq_intDBM_intro}
  \mathrm{d}Y^k_i(t) = \Biggl(\sum_{m\neq i} \frac{1-\theta}{Y^k_i(t)-Y^k_m(t)}-\sum_{m=1}^{k-1} \frac{1-\theta}
  {Y_i^k(t)-Y_m^{k-1}(t)}\Biggr)\,\mathrm{d}t + \mathrm{d}W_i^k(t), \quad 1\leq i\leq k\leq N,
\end{equation}
where $W_i^k$, $1\leq i\leq k\leq N$ are independent standard Brownian motions, possesses a unique
weak solution taking values in the cone \eq \overline{\mathcal{G}^N} =\left\{y=(y^k_i)_{1\leq
i\leq k\leq N}\in\rr^{N(N+1)/2}:\;y^{k-1}_{i-1}\leq y^k_i\leq y^{k-1}_i\right\} \en for any
initial condition $X(0)$ in the interior of $\overline{\mathcal{G}^N}$.
\end{theorem}

It would be interesting to extend all the above results to general $\theta>0$. We believe (but we
do not have a proof), that the identification of $Y^{mu}(t)$ with a solution of
\eqref{eq_intDBM_intro} is valid for any $\theta>1$ and that the identification of the projection
of $Y^{mu}(t)$ on level $N$ with $\beta=2\theta$ Dyson Brownian motion is valid for any
$\theta>0$. On the other hand, $Y^{mu}(t)$ can not be the solution to \eqref{eq_intDBM_intro} for
$\theta\le 1$. Indeed, we know that when $\theta=1$ the process $Y^{mu}(t)$ coincides with
reflected interlacing Brownian motions, which hints that one should introduce additional local
times terms in \eqref{eq_intDBM_intro}. In addition, the interpretation of the solution to
\eqref{eq_intDBM_intro} as a generalization of the one-dimensional Bessel process to a process in
the Gelfand-Tseitlin cone suggests that the corresponding process for $\theta<1$ is no longer a
semimartingale and should be defined and studied along the lines of \cite[Chapter XI, Exercise
(1.26)]{RY}.

\subsection{Our methods}

Our approach to the construction and study of the discrete approximating process
$X^{multi}_{disc}(s)$ is related to \emph{Jack symmetric polynomials}. Recall that Jack
polynomials $J_{\lambda}(x_1,x_2,\ldots,x_N;\theta)$, indexed by Young diagrams $\lambda$ and a
positive parameter $\theta$, are eigenfunctions of Sekiguchi differential operators (\cite{Sek},
\cite[Chapter VI, Section 10]{M}, \cite[Chapter 12]{For})
$$
 D(u;\theta)=\frac{1}{\prod_{i<j} (x_i-x_j)} \det\left[ x_i^{N-j}\left(x_i\,\frac{\partial}{\partial
 x_i}+(N-j)\theta+u\right)\right]_{i,j=1,2,\ldots,N}.
$$
One can also define $J_{\lambda}(x_1,x_2,\ldots,x_N;\theta)$ as limits of \emph{Macdonald
polynomials} $P_\lambda(\cdot;q,t)$ as $q,t\to 1$ in such a way that $t=q^{\theta}$ (see
\cite{M}). For the special values $\theta=1/2,\,1,\,2$ these polynomials are spherical functions
of Gelfand pairs $O(N)\subset U(N)$, $U(N)\subset U(N)\times U(N)$, $U(2N)\subset Sp(N)$,
respectively, and are also known as \emph{Zonal polynomials} (see e.g. \cite[Chapter 7]{M} and the
references therein). It is known that spherical functions of compact type (corresponding to the
above Gelfand pairs) degenerate to the spherical functions of Euclidian type, which in our case
are related to real symmetric, complex Hermitian and quaternionic Hermitian matrices, respectively
(see e.g.\ \cite[Section 4]{OO-Shifted} and the references therein). In particular, in the case
$\theta=1$ this is a manifestation of the fact that the tangent space to the unitary group $U(N)$
at identity can be identified with the set of Hermitian matrices. Due to all these facts it comes
at no surprise that Hermite $\beta$ ensembles can be obtained as limits of discrete probabilistic
structures related to Jack polynomials with parameter $\theta=\beta/2$.

\medskip

On the discrete level our construction of the multilevel stochastic dynamics is based on a
procedure introduced by Diaconis and Fill \cite{DF}, which has recently been used extensively in
the study of Markov chains on interlacing particle configurations (see e.g. \cite{BF},
\cite{BG_shuf}, \cite{B-Schur}, \cite{BG}, \cite{BC}). The idea is to use \emph{commuting Markov
operators} and conditional independence to construct a multilevel Markov chain with given single
level marginals. In our case these operators can be written in terms of Jack polynomials. In the
limit the commutation relation we use turns into the following statement, which might be of
independent interest.

\medskip

Let $P_N(t;\beta)$ denote the Markov transition operators of the $N$-dimensional Dyson Brownian
Motion of \eqref{eq_Dyson_BM} and let $L^N_{N-1}(\beta)$ denote the Markov transition operator
corresponding to conditioning the $(N-1)$-st level (that is,
$x^{N-1}_1,x^{N-1}_2,\ldots,x^{N-1}_{N-1}$) on the $N$-th level (that is,
$x^N_1,x^N_2,\ldots,x^N_N$) in the Hermite $\beta$ corners process.

\begin{proposition}[Corollary of Theorem \ref{Theorem_restriction_DBM}]
For any $\beta\ge 1$ the links $L^N_{N-1}(\beta)$ given by the stochastic transition kernels \eq
\begin{split}
\frac{\Gamma(N\beta/2)}{\Gamma(\beta/2)^N} \prod_{1\leq i<m\leq N-1} (x^{N-1}_m-x^{N-1}_i)
\prod_{i=1}^{N-1} \prod_{j=1}^N |x^N_j-x^{N-1}_i|^{\beta/2-1} \prod_{1\leq j<n\leq N}
(x^N_n-x^N_j)^{1-\beta}
\end{split}
\en intertwine the semigroups $P_N(t;\beta)$ and $P_{N-1}(t;\beta)$ in the sense that \eq
L^N_{N-1}(\beta) P_N(t;\beta)  = P_{N-1}(t;\beta)L^N_{N-1}(\beta),\quad t\geq0. \en
\end{proposition}

The latter phenomenon can be subsumed into a general theory of intertwinings for diffusions, which
for example also includes the findings in \cite{STW} and \cite{PS}.

\bigskip

Due to the presence of singular drift terms neither existence, nor uniqueness of the solution of
\eqref{eq_intDBM_intro} is straightforward. In the case of systems of SDEs with singular drift
terms one typically shows the existence and uniqueness of strong solutions by truncating the
singularity first (thus, obtaining a well-behaved system of SDEs) and by proving afterwards that
the solution cannot reach the singularity in finite time using suitable Lyapunov functions, see
e.g.\ \cite[proof of Proposition 4.3.5]{AGZ}. However, for $1<\theta<2$ the solutions of
\eqref{eq_intDBM_intro} \emph{do reach} some of the singularities. In a similar case of the
$\beta$--Dyson Brownian motion \eqref{eq_Dyson_BM} at $0<\beta<1$ the existence and uniqueness
theorem was established in \cite{CL} using multivalued $SDE$s, but we do not know how to use the
techniques of \cite{CL} in our multilevel setting. In addition, due to the intrinsic asymmetry
built into the drift terms the solution of \eqref{eq_intDBM_intro} seems to be beyond the scope of
the processes that can be constructed using Dirichlet forms (see e.g. \cite{PW} for Dirichlet form
constructions of symmetric diffusions with a singular drift at the boundary of their domain and
for the limitations of that method). Instead, by localizing in time, using appropriate Lyapunov
functions, and  by an application of Girsanov's Theorem, we are able to reduce
\eqref{eq_intDBM_intro} to a number of non-interacting Bessel process, for which the existence and
uniqueness of the solution of the corresponding SDE is well-known. This has an additional
advantage over Dirichlet form type constructions, since it allows to establish convergence to the
solution of \eqref{eq_intDBM_intro} via martingale problem techniques, which is how our proof of
Theorem \ref{Theorem_Limit_SDE} goes.

Note also that for $\theta=1$ ($\beta=2$) the interactions in the definition of
$X^{multi}_{disc}(s)$  become local and we studied the convergence of such dynamics to the process
of Warren in \cite{GS}. The proof in \cite{GS} was based on the continuity of a suitable Skorokhod
reflection map. For general values of $\theta$ neither the discrete dynamics, nor the continuous
dynamics can be obtained as the image of an explicitly known process under the Skorokhod
reflection map of \cite{GS}.

\subsection{Further developments and open problems.} It would be interesting to study the asymptotic
behavior of both the discrete and the continuous dynamics as the number of levels $N$ goes to
infinity. There are at least two groups of questions here.

\smallskip

The \emph{global} fluctuations of Dyson Brownian motions as $N\to\infty$ are known to be Gaussian
(see \cite[Section 4.3]{AGZ}); moreover, the limiting covariance structure can be described by the
Gaussian Free Field (see \cite{B-CLT}, \cite{BG-CLT}). In addition, the asymptotic fluctuations of
the Hermite $\beta$ corners processes of Definition \ref{defbetacorner} are also Gaussian and can
be described via the Gaussian Free Field (cf.\ \cite{BG-CLT}). This raises the question of whether
the 3-dimensional global fluctuations of the solution to \eqref{eq_intDBM_intro} are also
asymptotically (as $N\to\infty$) Gaussian and how the limiting covariance structure might look
like. A partial result in this direction was obtained for $\beta=2$ in \cite{BF}.

\smallskip

The \emph{edge} fluctuations (that is, the fluctuations of the rightmost particle as $N\to\infty$)
in the Hermite $\beta$ ensemble given by \eqref{eq_beta_Hermite} can be described via the
$\beta$-Tracy-Widom distribution (see \cite{RRV}). Moreover, in the present article we link the
Hermite $\beta$ ensemble to a certain discrete interacting particle system. This suggests trying
to find the $\beta$-Tracy-Widom distribution in the asymptotics of the edge fluctuations of that
interacting particle system or its simplified versions.

\medskip

\subsection{ Acknowledgements.} We would like to thank Amir Dembo for many fruitful discussions
and, in particular, for suggesting the use of a suitable Girsanov change of measure in the proof of
Theorem \ref{theorem_intDBM}. We also want to thank Alexei Borodin, Grigori Olshanski and Mark
Adler for helpful comments. V.G. was partially supported by RFBR-CNRS grant 11-01-93105.

\section{Discrete space dynamics via Jack polynomials} \label{sec:disc}

\subsection{Preliminaries on Jack polynomials}

Our notations generally follow the ones in \cite{M}. In what follows $\Lambda^N$ is the algebra of
symmetric polynomials in $N$ variables. In addition, we let $\Lambda$ be the algebra of symmetric
polynomials in countably many variables, or \emph{symmetric functions}. An element of $\Lambda$ is
a formal symmetric power series of bounded degree in the variables $x_1,x_2,\dots$. One way to
view $\Lambda$ is as an algebra of polynomials in Newton power sums $p_k=\sum_i (x_i)^k$. There
exists a unique canonical projection $\pi_N:\,\Lambda\to\Lambda_N$, which sets all variables
except for $x_1,x_2,\ldots,x_N$ to zero (see \cite[Chapter 1, Section 2]{M} for more details).

\medskip

A partition of size $n$, or a \emph{Young diagram} with $n$ boxes, is a sequence of non-negative
integers $\lambda_1\ge\lambda_2\ge\ldots\ge 0$ such that $\sum_i \lambda_i=n$. $|\lambda|$ stands
for the number of boxes in $\lambda$ and $\ell(\lambda)$ is the number of non-empty rows in
$\lambda$ (that is, the number of non-zero sequence elements $\lambda_i$ in $\lambda$). Let $\Y$
denote the set of all Young diagrams, also let $\Y^N$ denote the set of all Young diagrams
$\lambda$ with at most $N$ rows, i.e.\ such that $\lambda_{N+1}=0$. Typically, we will use the
symbols $\lambda$, $\mu$ for Young diagrams. We adopt the convention that the empty Young diagram
$\emptyset$ with $|\emptyset|=0$ also belongs to $\Y$ and $\Y^N$. For a box $\,\square=(i,j)$ of a
Young diagram $\lambda$ (that is, a pair $(i,j)$ such that $\lambda_i\ge j$), $a(i,j)$ and $l(i,j)$
are its arm and leg lengths:
$$
 a(i,j)=\lambda_i-j,\quad l(i,j)=\lambda'_j-i,
$$
where $\lambda'_j$ is the row length in the \emph{transposed} diagram $\lambda'$ defined by
$$
 \lambda'_j=|\{i:\lambda_i\ge j\}|.
$$
Further, $a'(i,j)$, $l'(i,j)$ are co-arm and co-leg lengths:
$$
 a'(i,j)=j-1,\quad l'(i,j)=i-1.
$$

\smallskip

We write $\J_\lambda(\,\cdot;\;\theta)$ for \emph{Jack polynomials}, which are indexed by Young diagrams $\lambda$
and positive reals $\theta$.
Many facts about these polynomials can be found in \cite[Chapter VI, Section 10]{M}. Note however
that in that book Macdonald uses the parameter $\alpha$ given by our $\theta^{-1}$. We use
$\theta$, following \cite{KOO}. $\J_\lambda$ can be viewed either as an element of the algebra
$\Lambda$ of symmetric functions in countably many variables $x_1,x_2,\ldots$, or (specializing
all but finitely many variables to zeros) as a symmetric polynomial in $x_1,x_2,\ldots,x_N$ from
the algebra $\Lambda^N$. In both interpretations the leading term of $\J_\lambda$ is given by
$x_1^{\lambda_1} x_2^{\lambda_2}\cdots x_{\ell(\lambda)}^{\lambda_{\ell(\lambda)}}$. When $N$ is
finite, the polynomials $\J_\lambda(x_1,\dots,x_N;\theta)$ are known to be the eigenfunctions of
the Sekiguchi differential operator:
\begin{multline}
\label{eq_Sekiguchi}
 \frac{1}{\prod_{i<j} (x_i-x_j)} \det\left[ x_i^{N-j}\left(x_i\,\frac{\partial}{\partial
 x_i}+(N-j)\theta+u\right)\right]_{i,j=1,2,\ldots,N} \J_\lambda(x_1,\dots,x_N;\theta)\\=\left(\prod_{i=1}^N (\lambda_i + (N-i)\theta +u)\right) \J_\lambda(x_1,\dots,x_N;\theta) .
\end{multline}
The eigenrelation \eqref{eq_Sekiguchi} can be taken as a definition for the Jack polynomials.
 We also need
dual polynomials $\JQ_\lambda$ which differ from $\J_\lambda$ by an \emph{explicit} multiplicative
constant:
$$
 \JQ_\lambda=\J_\lambda\cdot\prod_{\square\in\lambda} \frac{a(\square)+\theta l(\square)+\theta}{a(\square)+\theta
 l(\square)+1}.
$$

\smallskip

Next, we recall that \emph{skew Jack polynomials} $\J_{\lambda/\mu}$ can be defined as follows. Take two infinite sets
of variables $x$ and $y$ and consider a Jack polynomial $\J_\lambda(x,y;\theta)$. In particular, the latter is a symmetric
polynomial in the $x$ variables. The coefficients $\J_{\lambda/\mu}(y;\theta)$ in its decomposition in the linear basis
of Jack polynomials in the $x$ variables are symmetric polynomials in the $y$ variables and are referred to as skew Jack polynomials:
\begin{equation}
\label{eq_Skew_Jack}
 \J_\lambda(x,y;\theta) = \sum_{\mu} \J_{\mu}(x;\theta) \J_{\lambda/\mu}(y;\theta).
\end{equation}
Similarly, one writes
$$
 \JQ_\lambda(x,y;\theta) = \sum_{\mu} \JQ_{\mu}(x;\theta) \JQ_{\lambda/\mu}(y;\theta).
$$

Throughout the article the parameter $\theta$ remains fixed and, thus, we will often omit it from
the notations writing simply $\J_\lambda(x)$, $\JQ_\lambda(x)$, $\J_{\lambda/\mu}(x)$,
$\JQ_{\lambda/\mu}(x)$.

\smallskip

A \emph{specialization} $\rho$ is an algebra homomorphism from $\Lambda$ to the set of complex numbers.
We call the specialization $\rho$ which maps a polynomial to its free term the \emph{empty}
specialization. A specialization is called Jack-positive if its values on all (skew) Jack polynomials with a
fixed parameter $\theta>0$ are real and non-negative. The following statement gives a classification of all
Jack-positive specializations.

\begin{proposition}[\cite{KOO}] For any fixed $\theta>0$, Jack-positive specializations can be parameterized by triplets $(\alpha,\beta,\gamma)$,
where $\alpha$, $\beta$ are sequences of real numbers with
$$\alpha_1\ge\alpha_2\ge\ldots\ge 0,\quad \beta_1\ge\beta_2\ge\ldots\ge 0, \quad \sum_i
(\alpha_i+\beta_i)<\infty
$$
and $\gamma$ is a non-negative real number. The specialization corresponding to a triplet $(\alpha,\beta,\gamma)$ is given by its
values on Newton power sums $p_k$, $k\geq1$:
\begin{eqnarray*}
&&p_1\mapsto p_1(\alpha,\beta,\gamma)= \gamma+\sum_i (\alpha_i+\beta_i), \\
&&p_k\mapsto p_k(\alpha,\beta,\gamma)= \sum_i \alpha_i^k + (-\theta)^{k-1} \sum_i \beta_i^k, \quad k\ge 2.
\end{eqnarray*}
\end{proposition}


\begin{rmk} If all parameters of the specialization are taken to be zero, then we arrive at the empty
specialization. \end{rmk}

We prepare the following explicit formula for Jack-positive specializations for future use.

\begin{proposition}[{\cite[Chapter VI, (10.20)]{M}}]
\label{proposition_alpha}
 Consider the Jack--positive specialization $a^N$ with $\alpha_1=\alpha_2=\ldots=\alpha_N=a$ and
 all other parameters set to zero. We have
 $$
  \J_\lambda(a^N)=\begin{cases} a^{|\lambda|}\prod_{\square\in\lambda} \dfrac{N\theta+a'(\square)-\theta
  l'(\square)}{a(\square)+\theta l(\square) +\theta} & \mathrm{if\;}\ell(\lambda)\le N,\\
  0&\mathrm{otherwise.}
  \end{cases}
$$
\end{proposition}

Taking the limit $N\to\infty$ of specializations $\left(\frac{s}{N}\right)^N$ of Proposition
\ref{proposition_alpha} we obtain the following.

\begin{proposition}
\label{proposition_plancherel}
 Consider the Jack-positive specialization $\mathfrak{r}_s$ with $\gamma=s$ and
 all other parameters set to zero. We have
 $$
  \J_\lambda(\mathfrak{r}_s)=s^{|\lambda|} \theta^{|\lambda|} \prod_{\square\in\lambda} \frac{1}{a(\square)+\theta l(\square) +\theta}.
 $$
\end{proposition}

\begin{rmk} The specialization with finitely many equal parameters $\beta_i$ and all other
parameters set to zero also admits an explicit formula, but we do not need it here.
\end{rmk}

\subsection{Probability measures related to Jack polynomials} \label{Section_limit_at_fixed_time}

We start with the definition of Jack probability measures.

\begin{definition}
\label{def_Jack_mes} Given two Jack-positive specializations $\rho_1$ and $\rho_2$, the Jack
probability measure $\JM_{\rho_1;\rho_2}$ on $\Y$ is defined through
$$
   \JM_{\rho_1;\rho_2}(\lambda) = \dfrac{\J_\lambda(\rho_1)
   \JQ_\lambda(\rho_2)}{H_\theta(\rho_1;\rho_2)}
$$
 with the normalization constant being given by
$$
 H_{\theta}(\rho_1;\rho_2)=\exp\Big(\sum_{k=1}^{\infty} \frac{\theta}{k}\,p_k(\rho_1)p_k(\rho_2)\Big).
$$
\end{definition}

\begin{rmk}
The above definition makes sense only if $\rho_1$, $\rho_2$ are such that
$$
  \sum_\lambda \J_\lambda(\rho_1) \JQ_\lambda(\rho_2)<\infty,
$$
in which case the latter sum equals to the expression for $H_{\theta}(\rho_1;\rho_2)$ above due to
the Cauchy-type identity for Jack polynomials (see e.g. \cite[(10.4), Chapter VI, Section 10]{M}).
\end{rmk}
\begin{rmk}
 The construction of probability measures via specializations of symmetric polynomials
 was originally suggested by Okounkov in the context of \emph{Schur measures} \cite{Ok}. Recently
 similar constructions for more general polynomials have lead to many interesting results starting
 from the paper \cite{BC} by Borodin and Corwin. We refer to \cite[Introduction]{BCGS} for the
 chart of probabilistic objects which are linked to various degenerations of \emph{Macdonald
 polynomials}.
\end{rmk}

\smallskip

The following statement is a corollary of Propositions \ref{proposition_alpha} and
\ref{proposition_plancherel}.
\begin{proposition} \label{Proposition_prelimit_one_level}
Take specializations $1^N$ and $\mathfrak{r}_s$ of Propositions \ref{proposition_alpha} and
\ref{proposition_plancherel}. Then, $\JM_{1^N;\mathfrak{r}_s}(\lambda)$ vanishes unless
$\lambda\in\Y^N$ and in the latter case we have
\begin{equation}
\label{eq_prelimit_density}
 \JM_{1^N;\mathfrak{r}_s}(\lambda)=\exp(-\theta s N) s^{|\lambda|}\theta^{|\lambda|} \prod_{\square\in\lambda}
 \frac{N\theta+a'(\square)-\theta
  l'(\square)}{(a(\square)+\theta l(\square) +\theta)
 (a(\square)+\theta
 l(\square)+1)}.
\end{equation}
\end{proposition}

Next, we consider limits of the measures $\JM_{1^N;\mathfrak{r}_s}$ under a diffusive rescaling of
$s$ and $\lambda$. Define the open  Weyl chamber ${\mathcal{W}^N}=\{y\in\rr^N:\;y_1< y_2 <\ldots<
y_N\}$ and let $\overline{\mathcal{W}^N}$ be its closure.

\begin{proposition} \label{proposition_convergence_fixed_time} Fix some $N\in\nn$. Then, under the rescaling
$$
 s= \eps^{-1} \frac{t}{\theta}, \quad \lambda_i= \eps^{-1}t+ \eps^{-1/2} y_{N+1-i},\quad i=1,2,\ldots,N,
$$
the measures $\JM_{1^N;\mathfrak{r}_s}$ converge weakly in the limit $\eps\to0$ to the probability
measure with density \eq\label{eq_beta_Hermite_density} \frac{1}{Z} \prod_{i<j} (y_j-y_i)^{2\theta}
\prod_{i=1}^N \exp\left(- \frac{y_i^2}{2t}\right) \en on the closed Weyl chamber
$\overline{\mathcal{W}^N}$, where \eq\label{whatisZ} Z=t^{\theta
\frac{N(N-1)}{2}+\frac{N}2}(2\pi)^{N/2}\prod_{j=1}^{N} \frac{\Gamma(j\theta)}{\Gamma(\theta)}. \en
\end{proposition}

\begin{rmk} Note that we have chosen the notation is such a way that the row lengths $\lambda_i$ are non-increasing, while the continuous coordinates $y_i$ are non-decreasing in $i$.
\end{rmk}

\begin{proof}[Proof of Proposition \ref{proposition_convergence_fixed_time}]
We start by observing that \eqref{eq_beta_Hermite_density}, \eqref{whatisZ} define a probability
density, namely that the total mass of the corresponding measure is $1$. Indeed, the computation of
the normalization constant is a particular case of the \emph{Selberg integral} (see \cite{Sel},
\cite{For}, \cite{Meh}). Since $\JM_{1^N;\mathfrak{r}_s}$ is also a probability measure, it
suffices to prove that as $\eps\to 0$
$$
\JM_{1^N;\mathfrak{r}_s}(\lambda)=\eps^{N/2}\,\frac{1}{Z}\prod_{i<j} (y_j-y_i)^{2\theta}
\prod_{i=1}^N \exp\left(-\frac{y_i^2}{2t}\right)(1+o(1))
$$
with the error term $o(1)$ being uniformly small on compact subsets of ${\mathcal{W}^N}$. The
product over boxes in the first row of $\lambda$ in \eqref{eq_prelimit_density} is
\begin{eqnarray*}
\prod_{i=1}^{\lambda_1} (N\theta+i-1) \prod_{i=1}^{N-1} \prod_{j=\lambda_{i+1}+1}^{\lambda_i}
\frac{1}{(\lambda_1-j+\theta (i-1) +\theta)(\lambda_1-j+\theta (i-1)+1)} \\
=\frac{\Gamma(N\theta +\lambda_1)}{\Gamma(N\theta)}
\prod_{i=1}^{N-1} \frac{\Gamma(\lambda_1-\lambda_{i}+i \theta)}{\Gamma(\lambda_1-\lambda_{i+1}+i\theta)}
\frac{\Gamma(\lambda_1-\lambda_{i}+i \theta +1-\theta )}{\Gamma(\lambda_1-\lambda_{i+1}+i\theta+1-\theta)} \\
=\frac{\Gamma(\theta) }{\Gamma(N\theta) \Gamma((N-1)\theta+\lambda_1+1)} \prod_{i=2}^N
\Big( \frac{\Gamma(\lambda_1-\lambda_i+i \theta)}{\Gamma(\lambda_1-\lambda_i+(i-1)\theta)}
\times \frac{\Gamma(\lambda_1-\lambda_{i}+(i-1)\theta +1)}
{\Gamma(\lambda_1-\lambda_{i}+(i-2)\theta+1)} \Big) \\
\sim\frac{\Gamma(\theta)
}{\Gamma(N\theta) \Gamma((N-1)\theta +\lambda_1+1)} \prod_{i=1}^{N-1}
(\eps^{-1/2}(y_N-y_i))^{2\theta},
\end{eqnarray*}
where $A(\eps)\sim B(\eps)$ means here $\lim_{\eps\to 0} \frac{A(\eps)}{B(\eps)}=1$. Further,
\begin{multline*}
\frac{e^{-{t}{\eps^{-1}}}\,t^{\lambda_1}{\eps^{-\lambda_1}}}{\Gamma((N-1)\theta +\lambda_1+1)}
\sim e^{-{t}{\eps^{-1}}} \Big({t}{\eps^{-1}}\Big)^{{t}{\eps^{-1}}+{y_N}{\eps^{-1/2}}+1/2} \\
\quad\quad\quad\quad\quad\quad\quad\quad\quad
\times\frac{1}{\sqrt{2\pi}}\left(\dfrac{(N-1)\theta+{t}{\eps^{-1}}+{y_N}{\eps^{-1/2}}+1}{e}\right)
^{-(N-1)\theta-{t}{\eps^{-1}}-{y_N}{\eps^{-1/2}}-1} \\
\sim \frac{\big({t}{\eps^{-1}}\big)^{1/2-(N-1)\theta-1}}{\sqrt{2\pi}}\,e^{(N-1)\theta+1+{y_N}{\eps^{-1/2}}}
\left(1+\eps^{1/2}\,\frac{y_N}{t}+\eps\,\frac{(N-1)\theta+1}{t}\right)^
{-(N-1)\theta-{t}{\eps^{-1}}-{y_N}{\eps^{-1/2}}-1} \\
\sim \frac{\big({t}{\eps^{-1}}\big)^{1/2-(N-1)\theta+\theta}}{\sqrt{2\pi}} e^{(N-1)\theta+1+{y_N}{\eps^{-1/2}}}
\exp\Big(\Big(\eps^{1/2}\,\frac{y_N}{t}+\eps\,\frac{(N-1)\theta+1}{t}
-\eps\frac{y_N^2}{2\,t^2}\Big)\Big(-{t}{\eps^{-1}}-{y_N}{\eps^{-1/2}}\Big)\Big) \\
=  \eps^{1/2+(N-1)\theta}\,\frac{ t^{-1/2-(N-1)\theta}}{\sqrt{2\pi}}
\exp\left(-\frac{y_N^2}{2\,t}\right).
\end{multline*}
Therefore, the factors coming from the first row of $\lambda$ in \eqref{eq_prelimit_density} are asymptotically given by
\[
  \frac{\Gamma(\theta) }{\Gamma(N\theta)}\,\eps^{1/2}\,\frac{t^{-1/2-(N-1)\theta}}{\sqrt{2\pi}}
    \exp\left(-\frac{y_N^2}{2\,t}\right).  \prod_{i=1}^{N-1} ((y_N-y_i))^{2\theta}.
\]
Performing similar computations for the other rows we get
$$
\JM_{1^N;\mathfrak{r}_s}(\lambda)=\eps^{N/2} \prod_{j=1}^N
\frac{\Gamma(\theta)\,t^{-1/2-(j-1)\theta} }{ \Gamma(j\theta) \sqrt{2\pi}} \prod_{i<j}
(y_j-y_i)^{2\theta} \prod_{i=1}^N \exp\left(-\frac{y_i^2}{2\,t}\right)(1+o(1)),
$$
which finishes the proof.
\end{proof}

\smallskip

We now proceed to the definition of probability measures on multilevel structures associated with Jack polynomials.
Let $\GT^{(N)}$ denote the set of sequences of Young diagrams $\lambda^1,\lambda^2,\ldots,\lambda^N$ such that
$\ell(\lambda^i)\le i$ for every $i$ and the Young diagrams interlace:
$$
 \lambda^{i+1}_1\ge \lambda^i_1\ge \lambda^{i+1}_2\ge\dots\ge\lambda^{i}_i\ge\lambda^{i+1}_{i+1},\quad i=1,2,\ldots,N-1.
$$
We write $\lambda^i\prec \lambda^{i+1}$ for the latter set of inequalities.

\begin{definition}
\label{def_Jack_Gibbs}
 A probability distribution $P$ on arrays $(\lambda^1\prec\dots\prec\lambda^N)\in\GT^{(N)}$ is called
 a \emph{Jack--Gibbs} distribution, if for any $\mu\in\GT_N$, such that $P(\lambda^N=\mu)>0$,
 the conditional distribution of $\lambda^1\dots,\lambda^{N-1}$ given that $\lambda^N=\mu$ is
 \begin{equation}
 \label{eq_Jack_Gibbs}
  P(\lambda^1,\dots,\lambda^{N-1}\mid \lambda^N=\mu)=\frac{\J_{\mu/\lambda^{N-1}}(1)\J_{\lambda^{N-1}/\lambda^{N-2}}(1) \cdots
\J_{\lambda^2/\lambda^1}(1)\J_{\lambda^1}(1)}{J_{\mu}(1^N)}.
 \end{equation}
\end{definition}
\begin{rmk}
 The fact that \eqref{eq_Jack_Gibbs} is a well-defined probability distribution follows from the
 definition of skew Jack polynomials \eqref{eq_Skew_Jack}.
\end{rmk}
\begin{rmk} When $\theta=1$, \eqref{eq_Jack_Gibbs} means that the conditional distribution of $\lambda^1\dots,\lambda^{N-1}$
is \emph{uniform} on the polytope defined by the interlacing conditions.
\end{rmk}

One important example of a Jack--Gibbs measure is given by the following definition, see \cite{BG},
\cite{BP}, \cite{BC}, \cite{BCGS} for a review of related constructions in the context of Schur
and, more generally, Macdonald polynomials.

\begin{definition}
Given a Jack-positive specialization $\rho$, we define the ascending Jack process $\JM^{asc}_{\rho;N}$ as
the probability measure on $\GT^{(N)}$ given by
\begin{equation}
\label{eq_Jack_ascending}
 \JM^{asc}_{\rho;N}(\lambda^1,\lambda^2,\ldots,\lambda^N)
=\dfrac {\JQ_{\lambda^N}(\rho)\J_{\lambda^N/\lambda^{N-1}}(1)\cdots
\J_{\lambda^2/\lambda^1}(1)\J_{\lambda^1}(1)}{H_\theta(\rho;1^N)}.
\end{equation}
\end{definition}

\begin{rmk} The above definition makes sense only if $\rho$ is such that
$$
  \sum_\lambda \J_\lambda(\rho) \JQ_\lambda(
  1^N)<\infty,
$$
in which case the sum equals to $H_{\theta}(\rho;1^N)$ as in Definitions \ref{def_Jack_mes},
\ref{def_Jack_Gibbs}.
\end{rmk}

\begin{rmk} If $\rho$ is the empty specialization, then $\JM^{asc}_{\rho;N}$ assigns mass $1$ to
the single element of $\GT^{(N)}$ such that $\lambda^i_j=0$, $1\leq i\leq j\leq N$.
\end{rmk}

Define (open) Gelfand-Tsetlin cone via
\[
{\mathcal{G}^N}=\left\{y\in\rr^{N(N+1)/2}:\;y^{j-1}_{i-1}< y^j_i< y^{j-1}_i,\;1\leq i\leq j\leq
N\right\},
\]
and let $\overline{\mathcal{G}^N}$ be its closure. A natural continuous analogue of Definition
\ref{def_Jack_Gibbs} is:

\begin{definition}\label{thetaGibbsdef} An absolutely continuous (with respect to the Lebesgue measure)
probability distribution $P$ on arrays $y\in {\mathcal{G}^N}$ is called $\theta$--Gibbs, if the
conditional distribution of the first $N-1$ levels $y^k_i$, $1\leq i \leq k\leq N-1$ given the
$N$th level $y^N_1,\dots,y^N_N$ has density
\begin{multline}
\label{eq_theta_conditional}
 P\left(y^k_i,  1\leq i \leq k\leq N-1 \mid y^N_1,\dots, y^N_N\right)\\=
 \prod_{k=2}^N\left(
\frac{\Gamma(k\theta)}{\Gamma(\theta)^k} \prod_{1\leq i<m\leq k-1} (y^{k-1}_m-y^{k-1}_i)
\prod_{i=1}^{k-1} \prod_{j=1}^k |y^k_j-y^{k-1}_i|^{\theta-1} \prod_{1\leq j<n\leq k}
(y^k_n-y^k_j)^{1-2\theta}\right).
\end{multline}
\end{definition}

\begin{proposition}
 \label{proposition_convergence_to_theta_Gibbs}
Let $P(q)$, $q=1,2,\dots$ be a sequence of Jack--Gibbs measures on $\GT^{(N)}$ and let
$\{\lambda^k_i(q)\}$ be $P(q)$--distributed random element of $\GT^{(N)}$. Suppose that there
exist two sequences $a(q)$ and $b(q)$ such that $\lim_{q\to\infty} b(q)=+\infty$ and as
$q\to\infty$ the $N$--dimensional vector
$$
 \left(\frac{\lambda^N_N-a(q)}{b(q)}, \frac{\lambda^N_{N-1}-a(q)}{b(q)},\dots, \frac{\lambda^N_1-a(q)}{b(q)} \right)
$$
weakly converges to a random vector whose distribution is absolutely continuous with respect to
the Lebesgue measure. Then the full $N(N-1)$--dimensional vector
$$
 \left(\frac{\lambda^k_i-a(q)}{b(q)}\right),\quad 1\le i\le k\le N
$$
also weakly converges and the limit distribution is $\theta$--Gibbs.
\end{proposition}
\begin{proof}
 We show
first that \eqref{eq_theta_conditional} defines a probability measure. To this end, we use a
version of the Dixon-Anderson identity (see \cite{Dixon}, \cite{And}, \cite[Chapter 4]{For}),
which reads \eq\label{eq_DA}
\begin{split}
 \int\int\ldots\int \prod_{1\le i<j \le m} |u_i-u_j| \prod_{i=1}^m \prod_{j=1}^{m+1}
 |u_i-v_j|^{\theta-1}\,\mathrm{d}u_1\,\mathrm{d}u_2\,\ldots\,\mathrm{d}u_m \\
= \frac{\Gamma(\theta)^{m+1}}{\Gamma((m+1)\theta)}\prod_{1\le i <j \le m+1} |v_i-v_j|^{2\theta-1},
\end{split}
\en where the integration is performed over the domain
$$
 v_1<u_1<v_2<u_2<v_3<\ldots<u_m<v_{m+1}.
$$
Applying \eqref{eq_DA} sequentially to integrate the density in \eqref{eq_theta_conditional} with
respect to the variables $y^1_1$, then $y^2_1$, $y^2_2$ and so on, we eventually arrive at $1$.

\medskip

Next, we should check that the quantity in \eqref{eq_Jack_Gibbs} (written in the rescaled
coordinates) converges uniformly on compact subsets of the polytope defined by the interlacing
conditions to \eqref{eq_theta_conditional}. The evaluation of $\J_{\nu/\mu}(1)$ is known as the
\emph{branching rule for Jack polynomials} (see e.g. \cite[(7.14')]{M} or \cite[(2.3)]{OO}) and
reads
\begin{multline}
\label{eq_branching_rule} \psi_{\nu/\mu} := \J_{\nu/\mu}(1) = \prod_{1\le i\le j\le k-1}
 \frac{(\mu_i-\mu_j+\theta(j-i)+\theta)_{\mu_j-\nu_{j+1}}}{(\mu_i-\mu_j+\theta(j-i)+1)_{\mu_j-\nu_{j+1}}}
  \frac{(\nu_i-\mu_j+\theta(j-i)+1)_{\mu_j-\nu_{j+1}}}{(\nu_i-\mu_j+\theta(j-i)+\theta)_{\mu_j-\nu_{j+1}}},
\end{multline}
where $\ell(\nu)\le k$, $\mu\prec\nu$ and we use the Pochhammer symbol notation
$$
 (a)_n= a(a+1)\cdots (a+n-1).
$$
Thus, with the notation $ f(\alpha)=\frac{\Gamma(\alpha+1)}{\Gamma(\alpha+\theta)}$ we have
\[
 \J_{\lambda^{k}/\lambda^{k-1}}(1)= \prod_{1\le i\le j\le k-1} \frac{f(\lambda^{k-1}_i-\lambda^{k-1}_j+\theta(j-i))f(\lambda^k_i-\lambda^k_{j+1}+\theta(j-i))}{
 f(\lambda^{k-1}_i-\lambda^k_{j+1}+\theta(j-i))f(\lambda^k_i-\lambda^{k-1}_j+\theta(j-i))}.
\]
The asymptotics $f(\alpha)\sim\alpha^{1-\theta}$ as $\alpha\to\infty$ shows
\begin{multline*}
 \J_{\lambda^{k}/\lambda^{k-1}}(1)\sim \eps^{(k-1)(1-\theta)/2} \prod_{i=1}^{k-1} f(\theta(j-i))  \prod_{1\le i< j\le k-1} (y^{k-1}_j-y^{k-1}_i)^{1-\theta} \\
\cdot \prod_{1\le i< j\le k} (y^{k}_j-y^{k}_i)^{1-\theta} \prod_{i=1}^{k-1}\prod_{j=1}^k
|y^{k-1}_i-y^k_j|^{\theta-1}.
\end{multline*}
One obtains the proposition by putting together the asymptotics of the factors in
\eqref{eq_Jack_Gibbs}.
\end{proof}

As a combination of Propositions \ref{proposition_convergence_fixed_time} and
\ref{proposition_convergence_to_theta_Gibbs} we obtain the following statement.

\begin{corollary} \label{corollary_convergence_multi_level_fixed_time}
Fix some $N\in\nn$. Then, under the rescaling
$$
 s=\eps^{-1} \frac{t}{\theta},\quad \lambda_i^j=\eps^{-1}t+\eps^{-1/2}y_{j+1-i}^j,\quad
 1\leq i\leq j\leq N
$$
the measures $\JM^{asc}_{\mathfrak{r}_s;N}$ converge weakly in the limit $\eps\to0$ to the
probability measure on the Gelfand-Tsetlin cone ${\mathcal{G}^N}$ with density
\eq\label{eq_beta_Hermite_corners}
 \frac{1}{Z} \prod_{i<j} (y_j^N-y_i^N)
 \prod_{i=1}^N \exp\left(- \frac{(y_i^N)^2}{2\,t}\right)
 \prod_{n=1}^{N-1} \prod_{1\le i<j\le n} (y_j^n-y_i^n)^{2-2\theta} \\
\cdot\prod_{a=1}^n \prod_{b=1}^{n+1} |y^n_a-y^{n+1}_b|^{\theta-1},
\en
where
\eq\label{eq_beta_Hermite_corners2}
Z=t^{\theta N(N-1) +N/2}(2\pi)^{N/2}\prod_{j=1}^{N} \frac{(\Gamma(j\theta))^2}{\Gamma(\theta)^{j+1}}.
\en
\end{corollary}

\begin{rmk} Note that the probability measure of \eqref{eq_beta_Hermite_corners} is precisely the
Hermite $\beta=2\theta$ corners process with variance $t$ (see Definition \ref{defbetacorner}).
When $\theta=1$, the factors $(y_j^n-y_i^n)^{2-2\theta}$ and $|y^n_a-y^{n+1}_b|^{\theta-1}$ in
\eqref{eq_beta_Hermite_corners} disappear, and the conditional distribution of $y^1,
y^2,\dots,y^{N-1}$ given $y^N$ becomes \emph{uniform} on the polytope defined by the interlacing
conditions. This distribution is known to be that of eigenvalues of \emph{corners} of a random
Gaussian $N\times N$ Hermitian matrix sampled from the Gaussian Unitary Ensemble (see e.g.
\cite{Bar}). Similarly, for $\theta=1/2$ and $\theta=2$ one gets the joint distribution of the
eigenvalues of corners of the Gaussian Orthogonal Ensemble and the Gaussian Symplectic Ensemble,
respectively (see e.g. \cite{N}, \cite[Section 4]{OO-Shifted}).
\end{rmk}

\subsection{Dynamics related to Jack polynomials}

We are now ready to construct the stochastic dynamics related to Jack polynomials.
Similar constructions for Schur, $q$-Whittacker and Macdonald polynomials can be found in
\cite{BF}, \cite{B-Schur}, \cite{BG} and \cite{BC}.

\begin{definition} Given two specializations $\rho,\rho'$ define their union $(\rho,\rho')$ through
the formulas:
$$
 p_k(\rho,\rho')=p_k(\rho)+p_k(\rho'),\quad k\geq1,
$$
where $p_k$, $k\geq1$ are the Newton power sums as before.
\end{definition}

Let $\rho$ and $\rho'$ be two Jack-positive specializations such that
$H_\theta(\rho;\rho')<\infty$. Define matrices $p^{\uparrow}_{\lambda\to\mu}$ and
$p^{\downarrow}_{\lambda\to\mu}$ with rows and columns indexed by Young diagrams as follows:
\begin{equation}
 p^{\uparrow}_{\lambda\to\mu}(\rho;\rho')= \frac{1}{H_\theta(\rho;\rho')}\,
 \frac{\J_{\mu}(\rho)}{\J_\lambda(\rho)}\,\JQ_{\mu/\lambda}(\rho'), \quad \lambda, \mu\in\Y, \, \J_\lambda(\rho)\ne 0, \label{eq_p_up} \\
\end{equation}
\begin{equation}
 p^{\downarrow}_{\lambda\to\mu}(\rho;\rho')=
\frac{\J_{\mu}(\rho)}{\J_\lambda(\rho,\rho')}\,
 \J_{\lambda/\mu}(\rho'), \quad \quad\, \lambda,\mu\in\Y, \, \J_\lambda(\rho,\rho')\ne 0. \label{eq_p_down}
\end{equation}

The next two propositions follow from well-known properties of Jack polynomials (see
\cite{B-Schur}, \cite{BC}, \cite{BG} for analogous results in the cases of
 Schur, $q$-Whittacker and Macdonald polynomials).

\begin{proposition}
\label{Prop_p_is_stochastic} The matrices $p^{\uparrow}_{\lambda\to\mu}$ and
$p^{\downarrow}_{\lambda\to\mu}$ are stochastic, that is, all matrix elements are
non-negative, and for every $\lambda\in\Y$ we have
\begin{eqnarray}
 \sum\limits_{\mu\in\Y} p^{\uparrow}_{\lambda\to\mu}(\rho,\rho')=1,& \text{ if } \J_\lambda(\rho)\ne 0, \label{eq_x9} \\
 \sum\limits_{\mu\in\Y} p^{\downarrow}_{\lambda\to\mu}(\rho,\rho')=1, & \text{ if }
\J_\lambda(\rho,\rho')\ne 0. \label{eq_x10}
\end{eqnarray}
\end{proposition}

\begin{proposition} \label{Prop_p_agrees_with_Schur} For any $\mu\in\Y$ and any
Jack-positive specializations $\rho_1,\rho_2,\rho_3$ we have
\begin{eqnarray*}
&&\sum_{\lambda\in\Y\mid\, {\JM}_{\rho_1;\rho_2}(\lambda)\ne 0} {\JM}_{\rho_1;\rho_2}(\lambda)
p^\uparrow_{\lambda\to\mu}(\rho_2;\rho_3)=
 {\JM}_{\rho_1,\rho_3;\rho_2} (\mu) \\
&&\sum_{\lambda\in\Y\mid\, {\JM}_{\rho_1;\rho_2,\rho_3}(\lambda)\ne 0}
{\JM}_{\rho_1;\rho_2,\rho_3}(\lambda) p^\downarrow_{\lambda\to\mu}(\rho_2;\rho_3)=
 {\JM}_{\rho_1;\rho_2} (\mu).
\end{eqnarray*}
\end{proposition}

Let $X^N_{disc}(s)$, $s\geq0$ denote the continuous time Markov chain on $\Y^N$ with transition
probabilities given by $p^{\uparrow}(1^N;\mathfrak{r}_s)$, $s\geq0$ (and arbitrary initial
condition $X^N_{disc}(0)\in\Y^N$). We record the jump rates of $X^N_{disc}$ for later use.

\begin{proposition}
\label{prop_intensities_one_level} The jump rates of Markov chain $X^N_{disc}$ on $\Y^N$ are given
by
$$
 q_{\lambda\to\mu} =\begin{cases} \dfrac{\J_{\mu}(1^N)}{\J_\lambda(1^N)}\,\JQ_{\mu/\lambda}(\rho_1),& \mu=\lambda\sqcup
 \square,\\
 -\sum\limits_{\nu=\lambda\sqcup\square} q_{\lambda\to\nu},& \mu=\lambda,\\
 0,&\text{otherwise.}
  \end{cases}
$$
\end{proposition}

\begin{rmk} The jump rates $q_{\lambda\to\mu}$ are explicit. Indeed, $\J_\lambda(1^N)$ is
computed in Proposition \ref{proposition_alpha} and for $\mu=\lambda\sqcup (i,j)$
\begin{equation}
\label{eq_transition_explicit}
 \JQ_{\mu/\lambda}(\rho_1) =
   \theta \prod_{k=1}^{i-1} \frac{a(k,j)+\theta (i-k+1)}{a(k,j)+\theta (i-k)} \cdot
  \frac{a(k,j)+1+\theta (i-k-1)}{a(k,j)+1+\theta (i-k)},
\end{equation}
where the arm lengths are computed with respect to the (smaller) diagram $\lambda$ (see
\cite[Chapter VI, (6.20),(6.24), (7.13), (7.14)]{M}). In contrast, there are no fairly simple
 formulas for the transition probabilities $p^{\uparrow}(1^N;\mathfrak{r}_s)$, $s\geq0$ of
$X^N_{disc}$.
\end{rmk}

The following proposition turns out to be useful.

\begin{proposition}\label{Prop_sumpoisson} The process $|X^N_{disc}|:=\sum_{i=1}^N (X^N_{disc})_i$ is a
Poisson process with intensity $N\theta$.
\end{proposition}

\begin{proof}
According to Proposition \ref{prop_intensities_one_level} the process $|X^N_{disc}|$
increases by $1$ with rate
$$
\sum_{\mu=\lambda\sqcup\square} \dfrac{\J_{\mu}(1^N)}{\J_\lambda(1^N)} \JQ_{\mu/\lambda}(\rho_1).
$$
Note that when $\mu=\lambda\sqcup\square$, $\J_{\mu/\lambda}$ is a symmetric polynomial of degree $1$ with leading coefficient $\psi_{\mu/\lambda}$ defined in \eqref{eq_branching_rule}. Therefore,
$$
\JQ_{\mu/\lambda}(\rho_1)= \psi_{\mu/\lambda}\,\frac{\JQ_{\mu/\lambda}(\rho_1)}{\J_{\mu/\lambda}(\rho_1)}\,p_1(\rho_1)
= \psi_{\mu/\lambda} \frac{\JQ_{\mu/\lambda}(\rho_1)}{\J_{\mu/\lambda}(\rho_1)}.
$$
Now, an application of the \emph{Pierry rule} for Jack polynomials (see \cite[(6.24) and Section 10 in
 Chapter VI]{M}) yields that
$$
 \sum_{\mu=\lambda\sqcup\square} {\J_{\mu}(1^N)}\,\psi_{\mu/\lambda}
 \,\frac{\JQ_{\mu/\lambda}(\rho_1)}{\J_{\mu/\lambda}(\rho_1)}=\theta e_1(1^N)\J_\lambda(1^N)=N\theta \J_\lambda(1^N) .
$$
This finishes the proof.
\end{proof}

Proposition \ref{Prop_p_agrees_with_Schur} implies the following statement.

\begin{proposition}
Suppose that the initial condition $X^N_{disc}(0)$  is the empty Young diagram (so that
$\lambda_1=\lambda_2=\ldots=\lambda_N=0$). Then, for any fixed $s>0$ the law of $X^N_{disc}(s)$ is
given by $\JM_{1^N;\mathfrak{r}_s}$ (see Proposition \ref{Proposition_prelimit_one_level} for an
explicit formula).
\end{proposition}

Our next goal is to define a stochastic dynamics on $\GT^{(N)}$. The construction we use is
parallel to those of \cite{BF}, \cite{BG_shuf}, \cite{B-Schur}, \cite{BC}, \cite{BG}; it is
based on an idea going back to \cite{DF}, which allows to couple the dynamics of Young diagrams
of different sizes. We start from the degenerate discrete time dynamics $\lambda^0(n)=\emptyset$,
$n\in\nn_0$ and construct the discrete time dynamics of $\lambda^1,\lambda^2,\ldots,\lambda^N$
inductively. Given $\lambda^{k-1}(n)$, $n\in\nn_0$ and a Jack-positive specialization $\rho$
we define the process $\lambda^k(n)$, $n\in\nn_0$ with a given initial condition $\lambda^k(0)$
satisfying $\lambda^{k-1}(0)\prec\lambda^k(0)$ as follows. We let the distribution of
$\lambda^k(n+1)$ depend only on $\lambda^k(n)$ and $\lambda^{k-1}(n+1)$ and be given by
\begin{equation}
\label{eq_2lev_discrete_time}
 \P(\lambda^k(n+1)=\nu\mid \lambda^k(n)=\lambda,\, \lambda^{k-1}(n+1)=\mu)=\dfrac{\JQ_{\nu/\lambda}(\rho)\,\J_{\nu/\mu}(1)}
{\sum_{\kappa\in\Y} \JQ_{\kappa/\lambda}(\rho)\,\J_{\kappa/\mu}(1)}.
\end{equation}
Carrying out this procedure for $k=1,2,\ldots,N$ we end up with a discrete time Markov chain
$\hat X^{multi}_{disc}(n;\rho)$, $n\in\nn_0$ on $\GT^{(N)}$.

\begin{definition}
\label{Def_X_multi} Define the continuous time dynamics $X^{multi}_{disc}(s)$, $s\geq0$ on
$\GT^{(N)}$ with an initial condition $X^{multi}_{disc}(0)\in\GT^{(N)}$ as the limit
$$
\lim_{\eps\to 0} \hat X^{multi}_{disc}(\lfloor \eps^{-1}s \rfloor;\mathfrak{r}_\eps),
$$
where all dynamics $\hat X^{multi}_{disc}(\cdot;\mathfrak{r}_\eps)$ are started from the initial
condition $X^{multi}_{disc}(0)$ and the specialization $\mathfrak{r}_\eps$ is defined as in
Proposition \ref{proposition_plancherel}.
\end{definition}

\begin{rmk} Alternatively, we could have started from the specialization $\rho$ with a single
$\alpha$ parameter $\alpha_1=\eps$ and we would have arrived at the same continuous time dynamics.
Analogous constructions of the continuous time dynamics in the context of Schur and Macdonald polynomials
can be found in \cite{BF} and \cite{BC}.
\end{rmk}

\smallskip

Note that when $\lambda\subset\kappa$ the term $\JQ_{\kappa/\lambda}(\mathfrak{r}_\eps)$ is of
order $\eps^{|\kappa|-|\lambda|}$ as $\eps\to 0$. Therefore, the leading order term in the sum on
the right-hand side of \eqref{eq_2lev_discrete_time} comes from the choice $\kappa=\lambda$ unless
that $\kappa$ violates $\mu\subset\kappa$, in which case the leading term corresponds to taking
$\kappa=\mu$. Moreover, the first-order terms come from the choices $\kappa=\lambda\sqcup\square$
and the resulting terms turn into the jump rates of the continuous time dynamics. Summing up, the
continuous time dynamics $X^{multi}_{disc}(s)$, $s\geq0$ looks as follows: given the trajectory of
$\lambda^{k-1}$, a box $\square$ is added to the Young diagram $\lambda^k$ at time $t$ at the rate
\begin{equation}
\label{eq_jump_intensity}
 q(\square,\lambda^k(s-),\lambda^{k-1}(s))
=\JQ_{(\lambda^{k}(s-)\sqcup\square)/\lambda^{k}(s-)}(\rho_1)\,
\dfrac{\J_{(\lambda^{k}(s-)\sqcup\square)/\lambda^{k-1}(s)}(1)}{\J_{\lambda^{k}(s-)/\lambda^{k-1}(s)}(1)}.
\end{equation}
In particular, the latter jump rates incorporate the following push interaction: if the coordinates
of $\lambda^{k-1}$ evolve in a way which violates the interlacing condition $\lambda^{k-1}\prec
\lambda^{k}$, then the appropriate coordinate of $\lambda^k$ is pushed in the sense that a box is
added immediately to the Young diagram $\lambda^k$ to restore the interlacing. The factors on the
right-hand side of \eqref{eq_jump_intensity} are explicit. Indeed, the first one is given by
\eqref{eq_transition_explicit} and the second one by \eqref{eq_branching_rule}. Simulations of the
continuous time dynamics for $\theta=0.5$ and $\theta=2$ can be found at \cite{Han}.

\medskip

The following statement is proved by repeating the argument of \cite[Sections 2.2, 2.3]{BF}, see
also \cite{B-Schur}, \cite{BG}, \cite{BC}.

\begin{proposition}\label{prop_intertwining_disc} Suppose that $X^{multi}_{disc}(s)$, $s\geq0$ is started from
a random initial condition with a Jack--Gibbs distribution. Then
\begin{itemize}
\item
 Restriction of $X^{multi}_{disc}(s)$ to level $N$ coincides with $X^{N}_{disc}(s)$, $s\geq0$ started from the restriction
 to level $N$ of the  initial condition.
\item
 The distribution of $X^{multi}_{disc}(s)$ at time $s>0$ is again a Jack--Gibbs distribution. Moreover, if
 $X^{multi}_{disc}(0)$ has distribution $\JM^{asc}_{\rho;N}$, then $X^{multi}_{disc}(s)$ has distribution
 $\JM^{asc}_{\rho,\mathfrak{r}_s;N}$.
\end{itemize}
\end{proposition}

\begin{rmk} In fact, there is a way to generalize Proposition \ref{prop_intertwining_disc} to a
statement describing the restriction of our multilevel dynamics started from Jack--Gibbs initial
conditions to any monotonous space-time path (meaning that we look at level $N$ for some time,
then at level $N-1$ and so on). We refer the reader to \cite[Proposition 2.5]{BF} for a precise
statement in the setting of multilevel dynamics based on Schur polynomials.
\end{rmk}

\section{Convergence to Dyson Brownian Motion}
\label{Section_DBM}

The goal of this section is to prove that the Markov chain $X^N_{disc}(s)$ converges in the diffusive scaling limit to the Dyson Brownian Motion.

\smallskip

To start with, we recall the existence and uniqueness result for Dyson Brownian motions with $\beta>0$ (see e.g.\ \cite[Proposition 4.3.5]{AGZ} for the case $\beta\geq1$ and \cite[Theorem 3.1]{CL} for the case $0<\beta<1$).

\begin{proposition}\label{Prop_gen_DBM_thm}
For any $N\in\nn$ and $\beta>0$, the system of SDEs \eq\label{gen_DBM}
\mathrm{d}X_i(t)=\frac{\beta}{2}\,\sum_{j\neq i} \frac{1}{X_i(t)-X_j(t)}\,\mathrm{d}t +
\mathrm{d}W_i(t), \en $i=1,2,\ldots,N$, with $W_1,W_2,\ldots,W_N$ being independent standard
Brownian motions, has a unique strong solution taking values in the Weyl chamber $
\overline{\mathcal{W}^N}$ for any initial condition $X(0)\in\overline{\mathcal{W}^N}$. Moreover,
for all initial conditions, the stopping time \eq
\tau:=\inf\{t>0:\;X_i(t)=X_{i+1}(t)\;\mathrm{for\;some\;}i\} \en is infinite with probability $1$
if $\beta\geq1$ and finite with positive probability if $0<\beta<1$.
\end{proposition}

We write $D^N=D([0,\infty),\rr^N)$ for the space of right-continuous paths with left limits taking
values in $\rr^N$ and endow it with the usual Skorokhod topology (see e.g. \cite{EK}).

\begin{theorem} \label{theorem_DBM}
Fix $\theta\ge 1/2$ and let $\eps>0$ be a small parameter. Let $N$--dimensional stochastic
processes $Y^N_\eps(t)=(Y^N_\eps(t)_1,\dots,Y^N_\eps(t)_N)$ be defined through
$$
(Y^N_\eps(t))_i= \eps^{1/2}\Big((X^N_{disc})_{N+1-i}\Big(\eps^{-1}\frac{t}{\theta}\Big)-\eps^{-1}t\Big),
\quad i=1,\dots,N,
$$
where $(X^N_{disc})_i$ is $i$-th coordinate of the process $X^N_{disc}$. Suppose that as $\eps\to
0$, the initial condition $Y^N_\eps(0)$ converges to a point $Y(0)$ in the interior of
$\overline{\mathcal{W}^N}$. Then the process $Y^N_{\eps}(t)$ converges in the limit
$\eps\downarrow0$ in law in $D^N$ to the $\beta=2\theta$-Dyson Brownian motion, that is, to the
unique strong solution of \eqref{gen_DBM} with $\beta=2\theta$.
\end{theorem}
\begin{rmk}
We believe that Theorem \ref{theorem_DBM} should hold for any $\theta>0$. However, the case
$0<\theta< 1/2$ seems to be harder and one would need additional arguments for it.
\end{rmk}

Let us first present a plan of the proof of Theorem \ref{theorem_DBM}. In Step 1 we study the
asymptotics of jump rates of $X^N_{disc}$ in the scaling limit of Theorem \ref{theorem_DBM}. In
Step 2 we prove the tightness of processes $Y^N_\eps$ as $\eps\to 0$. In Step 3 we prove that
subsequential limits of $Y^N_\eps$ solve the stochastic differential equation \eqref{gen_DBM}.
This fact and the uniqueness of the solution to \eqref{gen_DBM} yield together Theorem
\ref{theorem_DBM}.


\subsection{Step 1: Rates}
$Y^N_\eps(t)$ is a continuous time Markov process with state space $\overline{\mathcal{W}^N}$, a
(constant) drift of $-\eps^{-1/2}$ in each coordinate and jump rates
$$
p^N_\eps(y,y',t)=\theta^{-1}\,\eps^{-1}\,q_{\frac{t}{\theta}\eps^{-1}+\hat{y}\eps^{-1/2}\,\rightarrow\,\frac{t}{\theta}\eps^{-1}+\hat{y'}\eps^{-1/2}},
$$
 where $\hat{y}$, $\hat{y'}$ are the vectors (which we view as rows of Young diagrams) obtained from $y$, $y'$ by reordering the coordinates
in decreasing order and intensities $q_{\lambda\to\mu}$ are given in Proposition
\ref{prop_intensities_one_level}. If we write $y'\approx_\eps y$ for vectors $y'$, $y$ which differ
in exactly one coordinate with the difference being $\eps^{1/2}$, then $p^N_\eps(y,y',t)=0$ unless
$y'\approx_\eps y$. As we will see, in fact, $p^N_\eps(y,y',t)$ does not depend on $t$.

Now, take two sequences $y'\approx_\eps y$ with $y'_{N+1-i}-y_{N+1-i}=\eps^{1/2}$ for some fixed
$i\in\{1,2,\ldots,N\}$. Define Young diagrams $\lambda$ and $\mu$ via
$\lambda_l=\frac{t}{\theta}\eps^{-1}+y_{N+1-l}\eps^{-1/2}$,
$\mu_l=\frac{t}{\theta}\eps^{-1}+y'_{N+1-j}\eps^{-1/2}$, then $\mu_i=\lambda_i+1$ and
$\mu_l=\lambda_l$ for $j\ne i$. Also set $j=\lambda_i+1$, that is, $\mu=\lambda\sqcup(i,j)$.
\begin{lemma}
\label{Lemma_expansion_of_rates_single}
For sequences $y'\approx_\eps y$ differing in the $(N+1-i)$-th coordinate as above, we have in the limit $\eps\to 0$:
$$
p^N_\eps(y,y',t)= \eps^{-1}+ \eps^{-1/2} \Biggl(\sum_{j\neq i}
\frac{\theta}{y_{N+1-i}-y_{N+1-j}}\Biggr) + O(1),
$$
where the error $O(1)$ is uniform on compact subsets of the open Weyl chamber $\mathcal W^N$.
\end{lemma}
\begin{proof}
Using Proposition \ref{prop_intensities_one_level} we have \eq\label{eq_intensity_single_analysis}
\begin{split}
p^N_\eps(y,y',t)=\frac{\eps^{-1}}{\theta}\,\frac{J_\mu(1^N)}{J_\lambda(1^N)}\tilde{J}_{\mu/\lambda}(\rho_1)
=\frac{\eps^{-1}}{\theta}\,\frac{\prod_{\square\in\mu} \dfrac{N\theta+a'(\square)-\theta
l'(\square)}{a(\square)+\theta l(\square) +\theta}}
 {\prod_{\square\in\lambda} \dfrac{N\theta+a'(\square)-\theta l'(\square)}{a(\square)+\theta l(\square) +\theta}} \\
 \cdot\,\theta \prod_{l=1}^{i-1} \frac{a(l,j)+\theta (i-l+1)}{a(l,j)+\theta (i-l)} \cdot
  \frac{a(l,j)+1+\theta (i-l-1)}{a(l,j)+1+\theta (i-l)} \\
=\frac{\eps^{-1}}{\theta}\,\prod_{k=1}^{j-1}
\frac{j-k-1+\theta(\lambda'_k-i+1)}{j-k+\theta(\lambda'_k-i+1)}
\cdot \prod_{l=1}^{i-1} \frac{\lambda_l-j+\theta(i-l)}{\lambda_l-j+\theta(i-l+1)} \\
\cdot\,((N-i+1)\theta+j-1) \prod_{l=1}^{i-1} \frac{\lambda_l-j+\theta (i-l+1)}{\lambda_l-j+\theta
(i-l)} \cdot
  \frac{\lambda_l-j+1+\theta (i-l-1)}{\lambda_l-j+1+\theta (i-l)} \\
=\frac{\eps^{-1}}{\theta}\,\prod_{k=1}^{j-1} \frac{j-k-1+\theta(\lambda'_k-i+1)}{j-k+\theta(\lambda'_k-i+1)}((N-i+1)\theta+j-1) \\
 \cdot\prod_{l=1}^{i-1} \frac{({y}_{N+1-l}-{y}_{N+1-i})\eps^{-1/2}+\theta(i-l-1)}{({y}_{N+1-l}-{y}_{N+1-i})\eps^{-1/2}+\theta (i-l)}.
\end{split}
\en
Now, for any $y$ the corresponding Young diagram $\lambda$ has $\lambda_N$ columns of length $N$,
$(\lambda_{N-1}-\lambda_N)$ columns of length $(N-1)$, $(\lambda_{N-2}-\lambda_{N-1})$ columns of
length $(N-2)$ etc. Therefore, for such $y$ the latter expression for $p^N_\eps(y,y',t)$ simplifies
to
\eq\label{eq_rates_single_level}
\begin{split}
&\eps^{-1}\prod_{r=0}^{N-i-1}
\frac{j-1-\lambda_{N-r}+\theta(N-i-r+1)}{j-1-\lambda_{N-r}+\theta(N-i-r)}
\prod_{l=1}^{i-1} \frac{({y}_{N+1-l}-{y}_{N+1-i})\eps^{-1/2}+\theta(i-l-1)}{({y}_{N+1-l}-{y}_{N+1-i})\eps^{-1/2}+\theta (i-l)} \\
&=\eps^{-1}\prod_{k=i+1}^N
\frac{({y}_{N+1-i}-{y}_{N+1-k})\eps^{-1/2}+\theta(k-i+1)}{({y}_{N+1-i}-{y}_{N+1-k})\eps^{-1/2}+\theta(k-i)}
\prod_{l=1}^{i-1}
\frac{({y}_{N+1-l}-{y}_{N+1-i})\eps^{-1/2}+\theta(i-l-1)}{({y}_{N+1-l}-{y}_{N+1-i})\eps^{-1/2}+\theta
(i-l)}\\
&=\eps^{-1}+ \eps^{-1/2} \left(\sum_{j\neq i}
 \frac{\theta}{y_{N+1-i}-y_{N+1-j}}\right) + O(1),
\end{split}
\en
with remainder $O(1)$ being uniform over $y$ such that $|y_{N+1-i}-y_{N+1-j}|>\delta$ for $j\ne i$
and a fixed $\delta>0$.
\end{proof}

\subsection{Step 2: Tightness}
\label{Section_tightness_single}

Let us show that the family $Y^N_\eps$, $\eps\in(0,1)$ is tight on $D^N$. To this end, we aim to
apply the necessary and sufficient condition for tightness of  \cite[Corollary 3.7.4]{EK} and need
to show that, for any fixed $t\geq0$, the random variables $Y^N_\eps(t)$ are tight on $\rr^N$ as
$\eps\downarrow 0$ and that for every $\Delta>0$ and $T>0$ there exists a $\delta>0$ such that
\[
\limsup_{\eps\downarrow0}\;\pp\Big(\sup_{0\leq s<t\leq T,t-s<\delta}
\big|(Y^N_\eps)_i(t)-(Y^N_\eps)_i(s)\big|>\Delta\Big)<\Delta,\;\; i=1,2,\ldots,N.
\]
Let us explain how to obtain the desired controls on $(Y^N_\eps(t))_+$ (the vector of positive
parts of the components of $Y^N_\eps(t)$) and
\eq\label{tight}
\sup_{0\leq s<t\leq T,t-s<\delta} \big((Y^N_\eps)_i(t)-(Y^N_\eps)_i(s)\big),\quad i=1,2,\ldots,N.
\en

To control $(Y^N_\eps(t))_+$ and the expressions in \eqref{tight}, we proceed by induction over the index
of the coordinates in $(Y^N_\eps(t))$. For the first coordinate $(Y^N_\eps)_1$ the explicit formula
\eqref{eq_rates_single_level} in step 1 shows that the jump rates of the process $(Y^N_\eps)_1$ are
bounded above by $\eps^{-1}$. Hence, a comparison with a Poisson process with jump rate
$\eps^{-1}$, jump size $\eps^{1/2}$ and drift $-\eps^{-1/2}$ shows that $((Y^N_\eps)_1(t))_+$ and
the expression in \eqref{tight} for $i=1$ behave in accordance with the conditions of Corollary
3.7.4 in \cite{EK} as stated above. Next, we consider $(Y^N_\eps)_i$ for some
$i\in\{2,3,\ldots,N\}$. In this case, the formula \eqref{eq_rates_single_level} in step 1 shows
that, whenever the spacing $(Y^N_\eps)_i-(Y^N_\eps)_{i-1}$ exceeds $\Delta/3$, the jump rate of
$(Y^N_\eps)_i$ is bounded above by
\eq \label{eq_rates_bound}
\eps^{-1}+\sum_{j=1}^{i-1} \frac{\theta\,\eps^{-1/2}}{(Y^N_\eps)_i(t)-(Y^N_\eps)_j(t)}+C(\Delta)
\leq\eps^{-1}+\frac{3(i-1)\theta}{\Delta}\,\eps^{-1/2}+C(\Delta).
\en

Let us show that $(Y^N_\eps)_i$ can be coupled with a process $R_\eps$ with jump size $\eps^{1/2}$,
jump rate given by the right-hand side of the last inequality and drift $-\eps^{-1/2}$, so that,
whenever $(Y^N_\eps)_i-(Y^N_\eps)_{i-1}$ exceeds $\Delta/3$ and $(Y^N_\eps)_i$ has a jump to the
right, the Poisson process has a jump to the right as well.

To do this, recall that (by definition) the law of the jump times of $Y^N_\eps$ can be described
as follows. We take $N$ independent exponential random variables $a_1,\dots,a_N$
with means $r_j(Y^N_\eps)$, $j=1,\dots,N$ given by \eqref{eq_intensity_single_analysis}. If we let
$k$ be the index for which $a_k=\min(a_1,\dots,a_N)$, then at time $a_k$ the $k$-th particle
(that is, $(Y^N_\eps)_k$) jumps. After this jump we repeat the procedure again to determine the next jump.

Let $M$ denote the right-hand side of \eqref{eq_rates_bound} and consider in each step an additional
independent exponential random variable $b$ with mean $M-r_i(Y^N_\eps)$ if
$(Y^N_\eps)_i-(Y^N_\eps)_{i-1}$ exceeds $\Delta/3$ and with mean $M$ otherwise. Now, instead
of considering $\min(a_1,\dots,a_N)$, we consider $\min(a_1,\dots,a_N,b)$. If the minimum is given by $b$,
then no jump happens and the whole procedure is repeated. Now, we \emph{define} the jump times of
process $R_\eps$ to be all times when the clock of the $i$-th particle rings provided that
$(Y^N_\eps)_i-(Y^N_\eps)_{i-1}$ exceeds $\Delta/3$, and also all times when the auxilliary random variable
$b$ constitutes the minimum. One readily checks that $R_\eps$ is given by a Poisson jump process of constant
intensity $M$ and drift $-\eps^{-1/2}$.

Thus, the convergence of $R_\eps$ to Brownian motion with drift $3(i-1)\theta/\Delta$ and the
necessary and sufficient conditions of Corollary 3.7.4 in \cite{EK} for them imply the
corresponding conditions for $((Y^N_\eps)_i(t))_+$ and the quantities in \eqref{tight} for the
value of $i$ under consideration.

\medskip

It remains to observe that $(Y^N_\eps(t))_-$ (the vector of negative parts of the components of
$Y^N_\eps$) and
\[
\sup_{0\leq s<t\leq T,t-s<\delta} -\big((Y^N_\eps)_i(t)-(Y^N_\eps)_i(s)\big),\quad i=1,2,\ldots,N
\]
can be dealt with in a similar manner (but considering the rightmost particle first and moving from
right to left). Together these controls yield the conditions of \cite[Corollary 3.7.4]{EK}.

\medskip

We also note that, since the maximal size of the jumps tends to zero as $\eps\downarrow 0$, any
limit point of the family $Y^N_\eps$, $\eps\in(0,1)$ as $\eps\downarrow 0$ must have continuous
paths (see e.g.\ \cite[Theorem 3.10.2]{EK}).

\begin{rmk}
 Note that in the proof of the tightness result the condition $\theta\ge 1/2$ is not used.
\end{rmk}

\subsection{Step 3: SDE for subsequential limits}

\label{Section_SDE_DBM_limits}

Throughout this section we let $Y^N$ be an arbitrary limit point of the family  $Y^N_\eps$ as
$\eps\downarrow 0$. Our goal is to identify $Y^N$ with the solution of \eqref{gen_DBM}. We pick a
sequence of $Y^N_\eps$ which converges to $Y^N$ in law, and by virtue of the Skorokhod Embedding
Theorem (see e.g. Theorem 3.5.1 in \cite{Du}) may assume that all processes involved are defined on
the same probability space and that the convergence holds in the almost sure sense. In the rest of
this section all the limits $\eps\to 0$ are taken along this subsequence.

\bigskip





Define the set of functions
\[
\mathcal{F}:=\big\{f\in C_0^\infty(\overline{\mathcal{W}^N})|\;\exists\,\delta>0:\;f(\mathbf
x)=0\;\,\mathrm{whenever}\;\,\mathrm{dist}(\mathbf
x,\partial\overline{\mathcal{W}^N})\leq\delta\big\},
\]
where $\partial\overline{\mathcal{W}^N}$ denotes the boundary of $\overline{\mathcal{W}^N}$ and
$\mathrm{dist}$ stands for the $L^\infty$ distance:
$$
 \mathrm{dist}(\mathbf
x,\partial\overline{\mathcal{W}^N})=\min_{i=1,\dots,N-1} |x_{i+1}-x_i|.
$$

For a function $f\in\mathcal F$ we consider the processes \eq\label{fMG}
M^f(t):=f(Y^N(t))-f(Y^N(0))-\int_0^t \sum_{1\leq i\neq j\leq N} \frac{\theta}{Y^N_i(s)-Y^N_j(r)}
f_{y_i}(Y^N(r))\,\mathrm{d}r -\frac{1}{2}\,\int_0^t \sum_{i=1}^N f_{y_iy_i}(Y^N(r))\,\mathrm{d}r.
\en Here, $f_{y_i}$ ($f_{y_iy_i}$ resp.) stands for the first (second resp.) partial derivative of
$f$ with respect to $y_i$.

In Step 3a we show that the processes in \eqref{fMG} are martingales and identify their quadratic
covariations. In step 3b we use the latter results to derive the SDEs for the processes $Y^N_1$,\dots, $Y^N_N$.

\medskip

{\noindent \bf \emph{Step 3a.}} We now fix an $f\in\mathcal{F}$ and consider the family of
martingales
\eq\label{fMGdisc}
\begin{split}
M^f_\eps(t):=f(Y^N_\eps(t))-f(Y^N_\eps(0))-\int_0^t \Big(\sum_{i=1}^N &-\eps^{-1/2}\,f_{y_i}(Y^N_\eps(r)) \\
&+\sum_{y'\approx_\eps Y^N(r)} p^N_\eps(Y^N_\eps(r),y',s)(f(y')-f(Y^N(s)))\Big)\,\mathrm{d}r,\quad
\eps>0.
\end{split}
\en
Lemma \ref{Lemma_expansion_of_rates_single} implies that the integrand in \eqref{fMGdisc}
behaves asymptotically as
\[
\frac{1}{2}\sum_{i=1}^N f_{y_iy_i}(Y^N_\eps(r))+\sum_{i=1}^N
b_i(Y^N_\eps(r)) f_{y_i}(Y^N_\eps(r)) + O(\eps^{1/2}),
\]
where $b_{i}(y)=\sum_{j\neq i} \frac{\theta}{y_{i}-y_j}$. Note also that, for any fixed function $f\in\mathcal{F}$, the error terms can be bounded uniformly for all sequences $y'\approx_\eps y$ as above, since $f$ and all its partial derivatives are bounded and vanish in the neighborhood of the boundary $\partial\overline{{\mathcal W}^N}$ of
$\overline{{\mathcal W}^N}$.

By taking the limit of the corresponding martingales $M^f_\eps$ for a fixed $f\in\mathcal{F}$ and
noting that their limit $M^f$ can be bounded uniformly on every compact time interval, we conclude
that $M^f$ must be a martingale as well.

\begin{lemma} \label{Lemma_covariance_single} For any function $g\in\mathcal F$, the quadratic variation of $M^g$ is given by
\[
\big\langle M^g\big\rangle(t) = \sum_{j=1}^N \int_0^t g_{y_j}(Y^N(r))^2\,\mathrm{d}r.
\]
\end{lemma}
\begin{proof}
We start by determining the limit of the quadratic variation processes $[M^g_\eps]$ of $M^g_\eps$ as $\eps\to0$. For each $\eps>0$ and $j=1,\dots,N$ define $\mathcal S^{j}_{\eps}$ as the (random) set of all
times when the $j$-th coordinate of $Y^N_\eps$ jumps. Note that the sets $\mathcal
S^{j}_{\eps}$ are pairwise disjoint and their union $\bigcup_{j=1}^N S^j_\eps$ is a Poisson point
process of intensity $\eps^{-1}N$ (see Proposition \ref{Prop_sumpoisson}).

Recall that the quadratic variation process $[M^g_\eps](t)$ of $M^g_\eps$ is given by the sum of
squares of the jumps of the process $M^g_\eps$ (see e.g.\ \cite[Proposition 8.9]{CT}) and conclude
\begin{equation}
\label{eq_quadratic_prelimit} [M^g_\eps](t)=\sum_{j=1}^N \sum_{r\in \mathcal S^{j}_{\eps} \cap
[0,t]} \big(\eps^{1/2}g_{y_j}(Y^N_\eps(r)+O(\eps)\big)^2,
\end{equation}
with uniform error term $O(\eps)$. Suppose that the support of $g\in\mathcal F$ is contained in the
set
$$
\{\y\in\overline{\mathcal W^N}\mid {\rm dist}(\y,\partial\overline{\mathcal{W}^N})\ge2\delta\}
$$
and consider new $N$ pairwise disjoint sets $\widehat{\mathcal S}^{j}_{\eps}$, $j=1,\dots,N$ satisfying
$\bigcup_{j=1}^N \widehat{\mathcal S}^{j}_{\eps} = \bigcup_{j=1}^N {\mathcal S}^{j}_{\eps}$ and defined
through the following procedure. Take any $r\in \bigcup_{j=1}^N S^{j}_{\eps}$ and suppose that $r\in \mathcal S^{k}_{\eps}$.
If ${\rm dist}(Y^N_\eps(r),\partial\overline{\mathcal{W}^N})\ge\delta$, then put $r\in \widehat
{\mathcal S}^k_\eps$. Otherwise, take an independent random variable $\kappa$ sampled from the
uniform distribution on the set $\{1,2,\dots,N\}$ and put $r\in \widehat {\mathcal S}^\kappa_\eps$.
The definition implies that, for small enough $\eps$,
\begin{equation}\label{eq_quadratic_prelimit2}
[M^g_\eps](t)=\sum_{j=1}^N \sum_{r\in \widehat {\mathcal S}^{j}_{\eps} \cap [0,t]} \big(\eps^{1/2}g_{y_j}(Y^N_\eps(r)+O(\eps)\big)^2,
\end{equation}
Now, take any two reals $a<b$. We claim that the sets $\widehat{\mathcal S}^{k}_{\eps}$ satisfy the
following property almost surely: \eq\label{eq_x11} \lim_{\eps\to 0} \eps\,|\widehat {\mathcal
S}^{k}_{\eps} \cap [a,b]| = b-a. \en Indeed, the Law of Large Numbers for Poisson Point Processes
implies \eq\label{eq_x12} \lim_{\eps\to 0} \eps\,\Big|\Big(\bigcup_{k=1}^N \widehat {\mathcal
S}^{k}_{\eps}\Big)\cap[a,b]\Big|=N(b-a). \en On the other hand, Lemma
\ref{Lemma_expansion_of_rates_single} implies the following \emph{uniform} asymptotics as $\eps\to
0$:
$$
\pp(t\in \widehat{\mathcal S}^{k}_\eps\mid\Theta_{<t},t\in\bigcup_{j=1}^N \widehat {\mathcal S}^{j}_\eps)=\frac{1}{N}+o(1),
$$
where $\Theta_{<t}$ is the $\sigma$-algebra generated by the point process $\widehat {\mathcal
S}^{k}_\eps$, $j=1,\dots,N$ up to time $t$. Therefore, the conditional distribution of $|\widehat
{\mathcal S}^{k}_{\eps} \cap [a,b]|$ given $|(\bigcup_{k=1}^N \widehat {\mathcal S}^{k}_{\eps})
\cap [a,b]|$ can be sandwiched between two binomial distributions with parameters
$\frac{1}{N}\pm C(\eps)$, where $\lim_{\eps\to 0} C(\eps)=0$ (see e.g. \cite[Lemma 1.1]{LSS}).
Now, \eqref{eq_x12} and the Law of Large Numbers for the Binomial Distribution imply \eqref{eq_x11}.

It follows that the sums in \eqref{eq_quadratic_prelimit2} approximate the corresponding integrals and we obtain
\eq \label{quad_var}
\lim_{\eps\downarrow0}\,[M^g_\eps](t) = \sum_{j=1}^N \int_0^t g_{y_j}(Y^N(r))^2\,\mathrm{d}r.
\en
Finally, note that for each $g\in\mathcal F$, both $M^g_\eps(t)^2$ and $[M^g_\eps](t)$ are uniformly
integrable on compact time intervals (this can be shown for example by another comparison with a
Poisson jump process). By \cite[Chapter 7, Problem 7]{EK} it follows that the process
$$
M^g(t)^2-\sum_{j=1}^N \int_0^t g_{y_j}(Y^N(r))^2\,\mathrm{d}r= \lim_{\eps\to 0}
\left(M^g_\eps(t)^2-[M^g_\eps](t)\right)
$$
is a martingale. The lemma readily follows.
\end{proof}

\medskip

{\bf \noindent\textit{Step 3b.}} We are now ready to derive the SDEs for the processes
$Y_1^N,\dots,Y_N^N$. Define the stopping times $\tau_\delta$, $\delta>0$ by
\eq\label{eq_stopping}
\tau_\delta = \inf\{t\geq0:\;Y_{i}^N(t)-Y_{i-1}^N(t)\le\delta\; \mathrm{for\;some\;}i\}
\wedge\inf\{t\geq0:\;|Y_i^N(t)|\geq 1/\delta\;\mathrm{for\;some\;}i\}.
\en

Our next aim is to derive the stochastic integral equations for the processes
$$
\big(Y^N_{j}(t\wedge\tau_\delta)\big)_{j=1,\dots,N}.
$$

Let $f_j$, $j=1,\dots,N$, be an arbitrary function from $\mathcal F$ such that $f_j(\y)=y_j$ for
$\y$ inside the box $|y_j|\le 1/\delta$ and such that ${\rm
dist}(\y,\partial\overline{\mathcal{W}^N})\ge \delta\}$. The results of Step 3a imply that the
processes $M^{{f}_j}$ are martingales. Note that the definition of stopping times $\tau_\delta$
imply that on the time interval $[0,\tau_\delta]$ the processes $M^{{f}_j}$ and $M^{y_j}$ almost
surely coincide. At this point we can use Lemma \ref{Lemma_covariance_single} to conclude that
$$
Y^N_j(t\wedge\tau_\delta))-(Y^N_j(0)) -\int_{0}^{t\wedge\tau_\delta} \sum_{n\neq j} \frac{\theta}{Y^N_j(r)-Y^N_n(r)} \,\mathrm{d}r,\qquad j=1,\dots,N
$$
are martingales with quadratic variations given by $t\wedge\tau_\delta$ and with the quadratic
covariation between any two of them being zero. We may now apply the Martingale Representation
Theorem in the form of \cite[Theorem 3.4.2]{KS} to deduce the existence of independent standard
Brownian motions $W_1,\ldots,W_N$ (possibly on an extension of the underlying probability space)
such that
\begin{equation}
Y^N_j(t\wedge\tau_\delta))-(Y^N_j(0)) -\int_{0}^{t\wedge\tau_\delta} \sum_{n\neq j}
\frac{\theta}{Y^N_j(r)-Y^N_n(r)} \,\mathrm{d}r = \int_{0}^{t\wedge\tau_\delta} \mathrm{d}W_j(s),\;\;\;
j=1,\dots,N.
\end{equation}
To finish the proof of Theorem \ref{theorem_DBM}, it remains to observe that Proposition \ref{Prop_gen_DBM_thm} implies
$$
 \lim_{\delta\to 0} \tau_\delta=\infty
$$
with probability one. \qedhere

\begin{rmk} An alternative way to derive the system of SDEs for $Y^N_j$ is to use \cite[Chapter 5,
Proposition 4.6]{KS}. This will be important for us in Section \ref{Section_multi_SDE}, since
generalization of Lemma \ref{Lemma_covariance_single} to the multilevel setting is not
straightforward.
\end{rmk}

\subsection{Zero initial condition}

A refinement of the proof of Theorem \ref{theorem_DBM} involving Proposition
\ref{proposition_convergence_fixed_time} allows us to deal with the limiting process, which is
started from $0\in\overline{{\mathcal W}^N}$.

\begin{corollary}\label{Corollary_zero_initial}
Fix $\theta\ge 1$. In the notations of Theorem \ref{theorem_DBM} and assuming the convergence of
the initial conditions to $0\in\overline{\mathcal{W}^N}$, the process $X^N_{disc}$ converges in
the limit $\eps\downarrow0$ in law in $D^N$ to the $\beta=2\theta$-Dyson Brownian motion started
from $0\in\overline{{\mathcal W}^N}$, that is, to the unique strong solution of \eqref{gen_DBM}
with $\beta=2\theta$ and $Y(0)=0\in\overline{{\mathcal W}^N}$.
\end{corollary}

\begin{proof}
Using Proposition \ref{proposition_convergence_fixed_time} and arguing as in the proof of Theorem
\ref{theorem_DBM} one obtains the convergence of the rescaled versions of the process $X^N_{disc}$
on every time interval $[t,\infty)$ with $t>0$ to the solution of \eqref{gen_DBM} starting
according to the initial distribution of \eqref{eq_beta_Hermite_density}. Since
\eqref{eq_beta_Hermite_density} converges to the delta--function at the origin as $t\to 0$, we
identify the limit points of the rescaled versions of $X^N_{disc}$ with the solution of
\eqref{gen_DBM} started from $0\in\overline{{\mathcal W}^N}$.
\end{proof}

\section{Existence and uniqueness for multilevel DBM} \label{Section_multi_unique}

The aim of this section is to prove an analogue of Proposition \ref{Prop_gen_DBM_thm} for the
multilevel Dyson Brownian motion.

\begin{theorem}\label{theorem_intDBM}
For any $N\in\nn$ and $\theta> 1$ (i.e.\ $\beta=2\theta> 2$), and for any initial condition $X(0)$
in the interior of $\overline{\mathcal{G}^N}$, the system of SDEs \eq\label{eq_intDBM_SDE}
\begin{split}
  \mathrm{d}X^k_i(t) =
  \Biggl( \sum_{m\neq i} \frac{1-\theta}{X^k_i(t)-X^k_m(t)}-\sum_{m=1}^{k-1} \frac{1-\theta}{X_i^k(t)-X_m^{k-1}(t)}\Biggr)\,\mathrm{d}t + \mathrm{d}W_i^k ,\;\;\;1\leq i\leq k\leq N
\end{split}
\en with $W_i^k$, $1\leq i\leq k\leq N$ being independent standard Brownian motions, possesses a
unique weak solution taking values in the Gelfand--Tsetlin cone $\GG$.
\end{theorem}
\begin{proof}

Given a stochastic process $X(t)$ taking values in $\overline{\mathcal{G}^N}$, for any fixed
$\delta>0$, let $\widehat \tau_{\delta}(X)$ denote
$$
 \widehat \tau_{\delta}[X]=\inf\{t\ge 0 \, :\, |X^k_i(t)-X^{k'}_{i'}(t)|\leq\delta,\, |k-k'|\le 1\},
$$
that is, the first time when two particles on adjacent levels are at the distance at most
$\delta$. Further, we
 define the stopping time $\tau_\delta[X]$ as the first time when \emph{three} particles on adjacent levels are at the
 distance at most $\delta$:
\begin{multline*}
\tau_{\delta}[X]=\inf\biggl\{t\ge 0\, : \, |X^k_i(t)-X^{k'}_{i'}(t)|\leq\delta,\,
|X^k_i(t)-X^{k''}_{i''}(t)|\leq\delta,\\  \text{ for } (k',i')\ne (k'',i'') \text{ such that }
|k-k'|=|k-k''|=1 \biggr\}.
\end{multline*}


Figure \ref{Figure_6events} shows schematically the six possible triplets of close to each other
particles at time $\tau_\delta$.

\begin{figure}[ht]
    \centering
  \includegraphics[width=15cm]{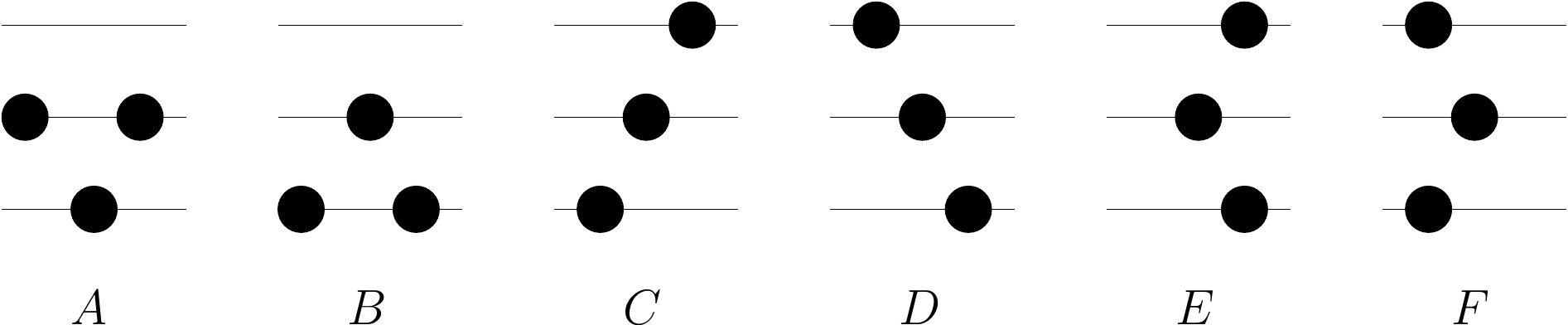}
    \caption{Six possible triplets of close to each other particles: One of these situations occurs
    at time $\tau_\delta$.}
    \label{Figure_6events}
\end{figure}

The following proposition will be proved in Section \ref{section_proof_of_uniqueness_integral}.

\begin{proposition} \label{Proposition_uniqueness_integral} For any $N\in\nn$, $\delta>0$ and $\theta> 1$ (i.e.\ $\beta=2\theta>2$), and for any initial condition $X(0)$ in the interior of $\overline{\mathcal{G}^N}$, the system of stochastic integral equations
\begin{equation}
\label{eq_stopped_integral}
\begin{split}
 X^k_i(t)-X^k_i(0)=
 \int_0^{t\wedge\tau_\delta[X]}\Big( \sum_{m\neq i} \frac{1-\theta}{X^k_i(s)-X^k_m(s)}-\sum_{m=1}^{k-1} \frac{1-\theta}{X_i^k(t)-X_m^{k-1}(t)}\Big)\,\mathrm{d}s
 + W_i^k(t\wedge\tau_\delta[X]), \\
1\leq i\leq k\leq N
\end{split}
\end{equation}
with $W_i^k$, $1\leq i\leq k\leq N$ being independent standard Brownian motions, possesses a unique weak solution.
\end{proposition}

In view of Proposition \ref{Proposition_uniqueness_integral} and the Kolmogorov extension theorem
(see e.g.\ \cite[Theorem 6.16]{Kal}), we can consider a product probability space which supports
independent  weak solutions of \eqref{eq_stopped_integral} for all $\delta>0$ and all initial
conditions in the interior of $\overline{\mathcal{G}^N}$. Choosing a sequence $\delta_l$,
$l\in\nn$ decreasing to zero, we can define on this space a process $X$ such that the law of
$X(t\wedge\tau_{\delta_1}[X])$, $t\ge0$ coincides with the law of the solution of
\eqref{eq_stopped_integral} with $\delta=\delta_1$ and initial condition $X(0)$, the law of
$X((\tau_{\delta_1}[X]+t)\wedge\tau_{\delta_2}[X])$, $t\ge0$ is given by the law of the solution
of \eqref{eq_stopped_integral} with $\delta=\delta_2$ and initial condition
$X(\tau_{\delta_1[X]})$ etc. The uniqueness part of Proposition
\ref{Proposition_uniqueness_integral} now shows that, for each $l\in\nn$, the law of
$X(t\wedge\tau_{\delta_l[X]})$, $t\ge0$ is that of the weak solution of
\eqref{eq_stopped_integral} with $\delta=\delta_l$. Since the paths of $X$ are continuous by
construction and hence $\lim_{l\to\infty} \tau_{\delta_l}[X]=\tau_0[X]$, we have constructed a
weak solution of the system
\begin{equation}
\label{eq_stopped_integral_0}
\begin{split}
 X^k_i(t)-X^k_i(0)=
 \int_0^{t\wedge\tau_0[X]}\left( \sum_{m\neq i} \frac{1-\theta}{X^k_i(s)-X^k_m(s)}-\sum_{m=1}^{k-1} \frac{1-\theta}{X_i^k(t)-X_m^{k-1}(t)}\right)\,\mathrm{d}s
 + W_i^k(t\wedge\tau_0[X]) ,\\
1\leq i\leq k\leq N
\end{split}
\end{equation}
with $W_i^k$, $1\leq i\leq k\leq N$ being independent standard Brownian motions as before. In addition, we note that the law of the solution to \eqref{eq_stopped_integral_0} is uniquely determined. Indeed, for any $\delta>0$, the process $X$ stopped at time $\tau_\delta$ would give a solution to \eqref{eq_stopped_integral}. Uniqueness of the latter for any $\delta>0$ now readily implies the uniqueness of the weak solution to \eqref{eq_stopped_integral_0}. At this point, Theorem \ref{theorem_intDBM} is a consequence of the following statement which will be proved in
Section \ref{section_stopping_time_infinity}.

\begin{proposition} \label{proposition_multlilevel_stopping} Suppose that $X(0)$ lies in the interior of the cone $\overline{\mathcal{G}^N}$ and let $X$ be a solution to \eqref{eq_stopped_integral_0}.
\begin{enumerate}[(a)]
\item If $\theta> 1$, then almost surely
 $\tau_0[X]=\infty.$
\item If $\theta \ge 2$, then almost surely
 $ \widehat \tau_{0}[X]=\infty.$
\end{enumerate}
\end{proposition}

\end{proof}

\subsection{Proof of Proposition \ref{Proposition_uniqueness_integral}}
\label{section_proof_of_uniqueness_integral} Our proof of Proposition
\ref{Proposition_uniqueness_integral} is based on an application of the Girsanov's theorem that
will dramatically simplify the SDE in consideration. We refer the reader to \cite[Section 3.5]{KS}
and \cite[Section 5.3]{KS} for general information about the Girsanov's theorem and weak solutions
to SDEs.

\smallskip

We start with the uniqueness part. Fix $N=1,2,\dots$, $\theta\ge 1$, $\delta>0$, and let $X$ be a
solution of \eqref{eq_stopped_integral}. Let $\mathcal I$ denote the set of $N(N-1)/2$ pairs
$(k,i)$, $k=1,\dots,N$, $i=1,\dots,k$ which represent different coordinates (particles) in the
process $X$. We will subdivide $\mathcal I$ into disjoint singletons and pairs of the neighboring
particles, that is, pairs of the form $((k,i),(k-1,i))$ or $((k,i),(k-1,i-1))$. We call any such
subdivision a \emph{pair-partition} of $\mathcal I$. An example is shown in Figure
\ref{Fig_pairs}.

\begin{figure}[h]
\begin{center}
 {\scalebox{0.8}{\includegraphics{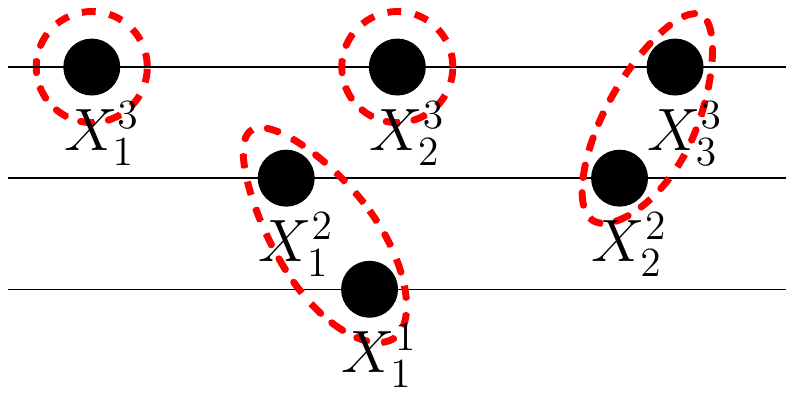}}}
\end{center}
\caption{A pair-partition with $N=3$, two pairs and two singletons.} \label{Fig_pairs}
\end{figure}

\begin{lemma}
\label{lemma_exist_stop}
 There exists a sequence of stopping times
 $0=\sigma_0\le \sigma_1\le \sigma_2\le \dots \le\tau_\delta[X]$ and (random) pair--partitions $A_1,
 A_2,\dots$ such that
 \begin{itemize}
 \item for any $n=1,2,\dots$, any $\sigma_{n-1}\le t < \sigma_n$, any two pairs
 $(k,i)$, $(k',i')$, $1\le i\le k\le N$, $1\le i'\le k'\le N$, $|k-k'|\le 1$,
 we have $|X^{k}_i(t)-X^{k'}_{i'}(t)|\ge \delta/2$ unless the pair $((k,i),(k',i'))$
 is one of the pairs of the pair-partition $A_n$, and
 \item for any $n=1,2,\dots$, either $\sigma_n=\tau_\delta$ or $|X^k_i(\sigma_{n+1})-X^k_i(\sigma_n)|\ge\delta/2$ for some $(k,i)$.
 \end{itemize}
\end{lemma}
\begin{proof} Define the (random) sets $\mathcal B^k_i$, $\mathcal D^k_i$ by setting
$$
 \mathcal B^k_i=\{ 0\le t\le \tau_\delta[X]\mid |X^k_i(t)-X^{k-1}_i(t)|\le\delta\},\quad \quad
 \mathcal D^k_i=\{ 0\le t\le \tau_\delta[X]\mid |X^k_i(t)-X^{k-1}_{i-1}(t)|\le\delta\}.
$$
Note that these sets are closed due to the continuity of the trajectories of $X$, which in turn is a consequence of \eqref{eq_stopped_integral}. Define $A(t;\delta)$ as a pair-partition such that pair $((k,i),(k-1,i))$ belongs to $A(t;\delta)$ iff $t\in \mathcal B^k_i$ and pair $((k,i),(k-1,i-1))$ belongs to $A(t;\delta)$ iff $t\in\mathcal D^k_i$. Similarly define $A(t,\delta/2)$. The definition of $\tau_\delta[X]$ implies that such pair-partitions
$A(t;\delta/2)\subset A(t;\delta)$ are well-defined for any $0\le t\le \tau_\delta[X]$.

Now we define $\sigma_n$ and $A_n$ inductively. First, set $\sigma_0=0$.  Further, for
$n=1,2,\dots$ let $A_n=A(\sigma_{n-1};\delta)$ and set $\sigma_n$ to be the minimal $t$ satisfying
$\tau_\delta[X]\ge t\ge \sigma_{n-1}$ and such that $A(t;\delta/2)$ has a pair which
$A(\sigma_{n-1};\delta)$ does not have. Since the sets $\mathcal B^k_i$, $\mathcal D^k_i$ are
closed, either such $t$ exists or no new pairs are added after time $\sigma_{n-1}$ and up to time
$\tau_\delta[X]$. In the latter case we set $\sigma_n=\tau_\delta$.
\end{proof}

Next, we fix a $T>0$, set $I_n=[\sigma_{n-1},\sigma_{n})$, $n=1,2,\dots$ and apply a Girsanov
change of measure (see e.g. \cite[Theorem 5.1, Chapter 3]{KS} and note that Novikov's condition as
in \cite[Corollary 5.13, Chapter 3]{KS} is satisfied due to the boundedness of the integrand in
the stochastic exponential) with a density of the form \eq \exp\Biggl(\sum_{n=0}^\infty
\sum_{1\leq i\leq k\leq N} \Big(\int_0^T b^k_{i,n}(t)\,\mathbf{1}_{I_n}(t)\,\mathrm{d}W^k_i(t) -
\frac{1}{2}\int_0^T (b^k_{i,n}(t))^2\,\mathbf{1}_{I_n}(t)\,\mathrm{d}t\Big)\Biggr), \en so that
under the new measure $\widetilde P$ for every fixed $k$, $i$, $n$ and $0\le t\le T$:
\begin{equation}
\label{eq_measure_changed_SDEs}
X^k_i(t\wedge\sigma_n)-X^k_i(t\wedge\sigma_{n-1})=\begin{cases}
\int\limits_{t\wedge\sigma_{n-1}}^{t\wedge\sigma_{n}}
\Big(\mathrm{d}\tilde{W}^k_i-\frac{(1-\theta)\,\mathrm{d}s}{X^k_i(s)-X^{k-1}_{i-1}(s)}
\Big), \;\;\; \text{ if }\;\;\; ((k,i),(k-1,i-1))\in A_n,\\
\int\limits_{t\wedge\sigma_{n-1}}^{t\wedge\sigma_{n}}\,
\Big(\mathrm{d}\tilde{W}^k_i-\frac{(1-\theta)\mathrm{d}s}{X^k_i(s)-X^{k-1}_i(s)}\Big),
\;\;\;\text{ if }\;\;\; ((k,i),(k-1,i))\in A_n,\\
\int\limits_{t\wedge\sigma_{n-1}}^{t\wedge\sigma_{n}}
\mathrm{d}\tilde{W}^k_i,\;\;\;\text{otherwise,}
\end{cases}
\end{equation}
where $\tilde{W}^k_i$, $1\leq i\leq k\leq N$ are independent standard Brownian motions under the
measure $\widetilde P$. We claim that the solution of the resulting system of SDEs
\eqref{eq_measure_changed_SDEs} is pathwise unique on $[0,\lim_{n\to\infty} \sigma_n)$ (that is,
for any two strong solutions of \eqref{eq_measure_changed_SDEs} adapted to the same Brownian
filtration, the quantities $\lim_{n\to\infty} \sigma_n$ for the two solutions will be the same
with probability one and the trajectories of the two solutions on $[0,\lim_{n\to\infty} \sigma_n)$
will be identical with probability one). Indeed, on each time interval $\sigma_{n-1}\le t \le
\sigma_n$, the system \eqref{eq_measure_changed_SDEs} splits into $|A_n|$ non-interacting systems
of SDEs each of which consists of one equation
\begin{equation}
\label{eq_decoupled_1}
 X^k_i(t\wedge\sigma_n)-X^k_i(t\wedge\sigma_{n-1})=
 \int\limits_{t\wedge\sigma_{n-1}}^{t\wedge\sigma_{n}} \mathrm{d}\tilde{W}^k_i
\end{equation}
if $(k,i)$ is a singleton in $A_n$, or of a system of two equations
\begin{equation}\label{eq_decoupled_2_1}
\begin{split}
\big(X^k_i(t\wedge\sigma_n)-X^{k-1}_{i'}(t\wedge\sigma_n)\big)
-\big(X^k_i(t\wedge\sigma_{n-1})-X^k_{i'}(t\wedge\sigma_{n-1})\big)
=\int\limits_{t\wedge\sigma_{n-1}}^{t\wedge\sigma_{n}}\big(\mathrm{d}\tilde{W}^k_i-\mathrm{d}\tilde{W}^{k-1}_{i'}\big) \\
-\int\limits_{t\wedge\sigma_{n-1}}^{t\wedge\sigma_{n}}\frac{(1-\theta)\,\mathrm{d}s}{X^k_i(s)-X^{k-1}_{i'}(s)},
\end{split}
\end{equation}
\begin{equation}
\label{eq_decoupled_2_2}
X^{k-1}_{i'}(t\wedge\sigma_n)-X^k_i(t\wedge\sigma_{n-1})=
\int\limits_{t\wedge\sigma_{n-1}}^{t\wedge\sigma_{n}} \mathrm{d}\tilde{W}^k_i
\qquad\qquad\qquad\qquad\qquad\qquad\qquad\qquad\qquad
\end{equation}
if $((k,i),(k-1,i'))$ is a pair in $A_n$. Therefore, one can argue by induction over $n$ and, once
pathwise uniqueness of the triplet $((X(t\wedge\sigma_{n-1}):\,t\ge0),\sigma_{n-1},A_{n-1})$ is
established, appeal to the pathwise uniqueness for \eqref{eq_decoupled_1},
\eqref{eq_decoupled_2_2} and \eqref{eq_decoupled_2_1} (the latter being the equation for the
Bessel process of dimension $\theta>1$, see \cite[Section 1, Chapter XI]{RY}) to deduce the
pathwise uniqueness of the triplet $((X(t\wedge\sigma_n):\,t\ge0),\sigma_n,A_n)$.

\medskip

The SDEs in \eqref{eq_measure_changed_SDEs} also allow us to prove the following statement.

\begin{lemma}
\label{lemma_stop_time_converge}
The identity $\lim_{n\to\infty} \sigma_n=\tau_\delta[X]$ holds with probability one.
\end{lemma}

\begin{proof}
It suffices to show that $\lim_{n\to\infty} \sigma_n\wedge T = \tau_\delta[X]\wedge T$ for any given $T>0$. Indeed, then
\[
\lim_{n\to\infty} \sigma_n\geq\lim_{T\to\infty} \lim_{n\to\infty} \sigma_n\wedge T = \lim_{T\to\infty} \tau_\delta[X]\wedge T = \tau_\delta[X]
\]
and $\lim_{n\to\infty} \sigma_n\leq\tau_\delta[X]$ holds by the definitions of the stopping times involved. If for some $n$ we have $\tau_\delta[X]\wedge T=\sigma_n\wedge T$, then we are done. Otherwise, $\sigma_n<T$ for all $n$ and the definition in Lemma \ref{lemma_exist_stop} shows that $|X^k_i(\sigma_{n+1})-X^k_i(\sigma_n)|\ge\delta/2$ for some $(k,i)$. In addition, \eqref{eq_measure_changed_SDEs} yields that, under the measure $\tilde{P}$, $|X^k_i(\sigma_{n+1})-X^k_i(\sigma_n)|$ is bounded above by the sum of absolute values of the increments of at most two Brownian motions and one Bessel process in time $(\sigma_{n+1}-\sigma_n)$. Since the trajectories of such processes are uniformly continuous on the compact interval $[0,T]$ with probability one, there exist two constants $c>0$ and $p>0$ such that
 $\tilde{P}(\sigma_{n+1}-\sigma_n>c)>p$. Consequently, $\sigma_n/c$ stochastically dominates a binomial random variable $Bin(n,p)$. In view of the law of large numbers for the latter, this is a contradiction to $\sigma_n<T$ for all $n$.
\end{proof}

\medskip

Now, we make a Girsanov change of measure back to the original probability measure and conclude
that the joint law of $X(t\wedge\tau_\delta\wedge T)$, $t\geq0$, $\sigma_n\wedge T$ and
$\tau_\delta\wedge T$ under the original probability measure is determined by such law under the
measure $\widetilde P$ (the justification for this conclusion can be found for example in the
proof of \cite[Proposition 5.3.10]{KS}). Since the latter is uniquely defined (by the law of the
solution to \eqref{eq_measure_changed_SDEs}), so is the latter. Finally, since $T>0$ was
arbitrary, we conclude that the joint law of $X(t\wedge\tau_\delta[X])$, $t\geq0$ and
$\tau_\delta[X]$ is uniquely determined.

\bigskip

To construct a weak solution to \eqref{eq_stopped_integral} we start with a probability space
$(\Omega,{\mathcal F},\pp)$ that supports a family of independent standard Brownian motions
$\tilde W^k_i$, $1\leq i\leq k\leq N$. In addition, we note (see \cite[Section XI]{RY} for a
proof) that to each pair of Brownian motions of the form $(\tilde W^k_i,\,\tilde W^{k-1}_{i-1})$
or $(\tilde W^k_i,\,\tilde W^{k-1}_i)$ and all initial conditions we can associate the unique
strong solutions of the SDEs
\begin{eqnarray}
\mathrm{d}R^{k,-}_i(t)=\frac{\theta-1}{R^{k,-}_i(t)}\,\mathrm{d}t+\mathrm{d}\tilde W^k_i(t)-\mathrm{d}\tilde W^{k-1}_{i-1}(t) ,\\
\mathrm{d}R^{k,+}_i(t)=\frac{\theta-1}{R^{k,+}_i(t)}\,\mathrm{d}t+\mathrm{d}\tilde
W^k_i(t)-\mathrm{d}\tilde W^{k-1}_i(t),
\end{eqnarray}
defined on the same probability space.

We will now construct $N(N-1)/2$-dimensional process $X(t)$, $t\ge 0$, stopping times $\tau_\delta$, $\sigma_n$, $n=0,1,2,\dots$ and pair-partitions $A_n$, $n=0,1,2,\dots$ which satisfy the conditions of Lemma \ref{lemma_exist_stop} and the system of equations \eqref{eq_measure_changed_SDEs}.

The construction proceeds for each $\omega\in\Omega$ independently, and is inductive. If the
initial condition $X(0)$ is such that $\tau_\delta[X]=0$, then there is nothing to prove.
Otherwise, we set $\sigma_0=0$ and $A_1=A(0;\delta)$ (see the proof of Lemma
\ref{lemma_exist_stop} for the definition of $A(t;\delta)$). Define $\hat X$ as the unique strong
solution to
\begin{equation}
\label{eq_measure_changed_SDEs2} \hat X^k_i(t)-\hat X^k_i(0)=\begin{cases} \int\limits_{0}^{t}
\Big(\mathrm{d}\tilde{W}^k_i-\frac{(1-\theta)\,\mathrm{d}s}{\hat X^k_i(s)-\hat X^{k-1}_{i-1}(s)}
\Big),\;\;\; \text{ if }\;\;\; ((k,i),(k-1,i-1))\in A_1,\\
\int\limits_{0}^{t}\, \Big(\mathrm{d}\tilde{W}^k_i-\frac{(1-\theta)\mathrm{d}s}{\hat X^k_i(s)-\hat
X^{k-1}_i(s)}\Big),
\;\;\;\text{ if }\;\;\; ((k,i),(k-1,i))\in A_1,\\
\int\limits_{0}^{t} \mathrm{d}\tilde{W}^k_i, \;\;\;\text{ otherwise,}
\end{cases}
\end{equation}
with initial condition $\hat X(0)=X(0)$.

Now we can define $\sigma_1$ as in Lemma \ref{lemma_exist_stop} but with $\hat X(t)$ used instead
of $X(t)$. After this we set $X(t)$ to be equal to $\hat X(t)$ on the time interval
$[0,\sigma_1]$. We further define $A_2=A(\sigma_1;\delta)$ and repeat the above procedure to
define $X(t)$ on the time interval $[\sigma_1,\sigma_2]$. Iterating this process and using the
result of Lemma \ref{lemma_stop_time_converge} we define $X(t)$ up to time $\tau_\delta[X]$. We
extend it to all $t\ge0$ by setting $X^k_i(t)=X^k_i(\tau_\delta[X])$ for $t>\tau_\delta[X]$.

\medskip

Next, we apply the Girsanov Theorem as in the uniqueness part to conclude that, for each $T>0$,
there exists a probability measure $\qq_T$ which is absolutely continuous with respect to $\pp$ and
such that the representation
\begin{equation*}
\begin{split}
X^k_i(t\wedge T)-X^k_i(0) =\int_0^{t\wedge T\wedge\tau_\delta[X]} \Biggl(\sum_{m\neq i}
\frac{1-\theta}{X^k_i(s)-X^k_m(s)}
-\sum_{m=1}^{k-1} \frac{1-\theta}{X^k_i(s)-X^{k-1}_m(s)}\Biggr)\,\mathrm{d}s \\
+ W^k_i(t\wedge T\wedge\tau_\delta[X]), \quad 1\leq i\leq k\leq N
\end{split}
\end{equation*}
holds with $W^k_i$, $1\leq i\leq k\leq N$ being independent standard Brownian motions under
$\qq_T$.

\medskip

Finally, replacing $T$ by a sequence $T_n\uparrow\infty$ and using the Kolmogorov Extension
Theorem (see e.g.\ \cite[Theorem 6.16]{Kal}) and note that the consistency condition is satisfied
due to the uniqueness of the solution to \eqref{eq_stopped_integral}), we deduce the existence of
processes $X^k_i$, $1\leq i\leq k\leq N$ on some probability space solving
\eqref{eq_stopped_integral}.

\subsection{Proof of Proposition \ref{proposition_multlilevel_stopping}}
\label{section_stopping_time_infinity}

We start with a version of Feller's test for explosions that will be used below (see e.g. \cite[Section 5.5.C]{KS} and the references therein for related results).

\begin{lemma}\label{hit0lemma}
Let $Z$ be a one-dimensional continuous semimartingale satisfying $Z(0)>0$ and \eq \forall\,0\leq
t_1<t_2:\quad Z(t_2)-Z(t_1)=b(t_2-t_1)+M(t_2)-M(t_1) \en with a constant $b>0$ and a local
martingale $M$. If the quadratic variation of $M$ satisfies \eq\label{quadvarcontr} \forall\,0\leq
t_1<t_2:\quad \left\langle M\right\rangle(t_2)-\left\langle M\right\rangle(t_1) \leq
2b\,\int_{t_1}^{t_2} Z(t)\,\mathrm{d}t, \en then the process $Z$ does not reach zero in finite
time with probability one.
\end{lemma}

\begin{proof} We fix two constants $0<l_1<Z(0)<L_1<\infty$ and let $\tau_{l_1,L_1}$ be the first time that $Z$ reaches $l_1$ or $L_1$. Next,
we apply It\^o's formula (see e.g.\ \cite[Section 3.3.A]{KS}) to obtain \eq\label{lnIto} \ln
Z(t\wedge\tau_{l_1,L_1}\wedge\zeta)-\ln Z(0)=\int_0^{t\wedge\tau_{l_1,L_1}\wedge\zeta}
\Big(\frac{b\,\mathrm{d}s}{Z(s)} - \frac{\mathrm{d}\langle
M\rangle(s)}{2\,Z(s)^2}\Big)+\int_0^{t\wedge\tau_{l_1,L_1}\wedge\zeta} \frac{\mathrm{d}M(s)}{Z(s)}
\en for any stopping time $\zeta$. By \eqref{quadvarcontr}, the first integral  in \eqref{lnIto}
takes non-negative values. Hence, picking a localizing sequence of stopping times $\zeta=\zeta_l$
for the local martingale given by the second integral in \eqref{lnIto}, taking the expectation in
\eqref{lnIto} and passing to the limit $l\to\infty$, we obtain
\[
\ev[\ln Z(t\wedge\tau_{l_1,L_1})]\geq \ln Z(0).
\]
Now, Fatou's Lemma and $Z(t\wedge\tau_{l_1,L_1})\leq L_1$ yield the chain of estimates
\[
\ln Z(0)\leq\ev[\limsup_{t\to\infty} \ln Z(t\wedge\tau_{l_1,L_1})]\leq p_{l_1}\ln l_1+(1-p_{l_1})\ln L_1
\]
where $p_{l_1}=P(\limsup_{t\to\infty} \ln Z(t\wedge\tau_{l_1,L_1})=l_1)$. Consequently,
\[
p_{l_1}\leq\frac{\ln L_1-\ln Z(0)}{\ln L_1 - \ln l_1}.
\]
The lemma now follows by taking the limit $l_1\downarrow0$.
\end{proof}

\medskip

We will now show that $\tau_0[X]=\infty$ for the solution of \eqref{eq_stopped_integral} with
initial condition $X(0)$ such that $\tau_0[X]>0$ (in particular, this includes the case that
$X(0)$ belongs to the interior of $\overline{\mathcal{G}^N}$). Recall that $\tau_0$ was defined as
the first time when one of the events in Figure \ref{Figure_6events} with $\delta=0$ occurs. We
will show that neither of the cases $A-F$ in Figure \ref{Figure_6events} can occur in finite time.
We will argue by the induction to show that none of these events happen on the first $k$ levels
for $k=1,2,\dots,N$.

First, we concentrate on the cases $A$ and $B$.

\begin{lemma} \label{lemma_AB_follows_CDEF} An event of the form $X^k_i(t)=X^k_{i+1}(t)$ cannot occur in finite time without one
of the events \eq\label{events} X^{k-1}_i(t)-X^{k-2}_{i-1}(t)=0,\quad X^{k-2}_i(t)-X^{k-1}_i(t)=0
\en occurring at the same time.
\end{lemma}

\begin{proof}
If the statement of the lemma was not true, then the continuity of the paths of the particles would allow us to find stopping times $\sigma$, $\sigma'$ similar to the ones introduced in Section \ref{section_proof_of_uniqueness_integral} and a real number $\kappa>0$ such that $\sigma<\sigma'$ with probability one, the spacings in \eqref{events} are at least $\kappa$ during the time interval $[\sigma,\sigma']$ and the event $X^k_i(t)=X^k_{i+1}(t)$ occurs for the first time at time $\sigma'$. Moreover, the interlacing condition and the induction hypothesis imply together that $[\sigma,\sigma']$ and $\kappa$ can be chosen such that the spacings
\[
X^k_i-X^{k-1}_{i-1},\quad X^{k-1}_{i+1}-X^k_{i+1}
\]
do not fall below $\kappa$ on $[\sigma,\sigma']$ (otherwise at least one of the events
$X^{k-1}_{i-1}(t)=X^{k-1}_i(t)$ or $X^{k-1}_i(t)=X^{k-1}_{i+1}(t)$ would have occurred at time
$\sigma'$ in contradiction to the induction hypothesis). The described inequalities are shown in
Figure \ref{Figure_decoup}.

\begin{figure}[h]
\includegraphics[width=8cm]{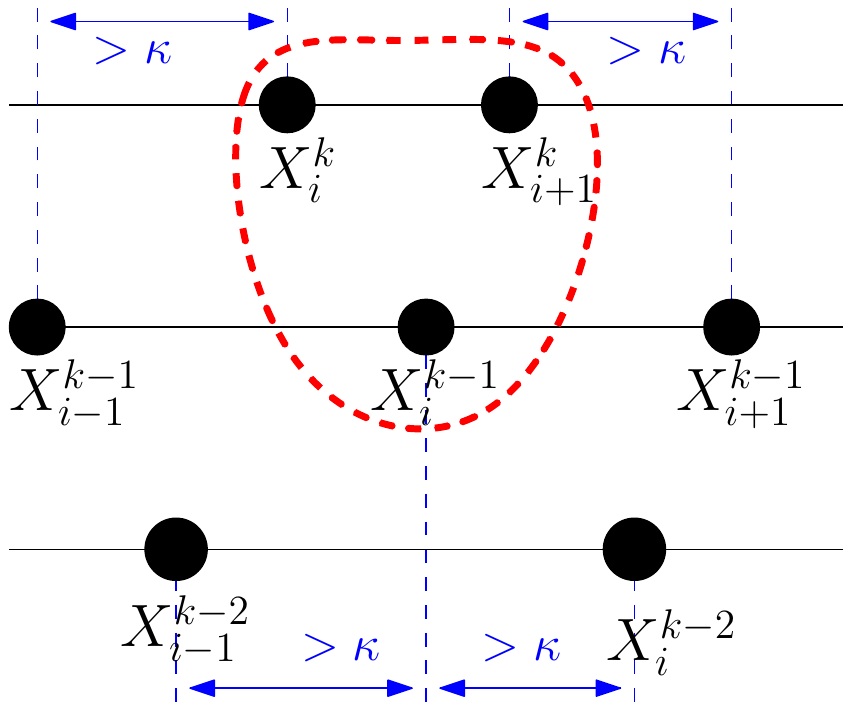}
\caption{Decoupling of the particles $X^k_i$, $X^k_{i+1}$, $X^{k-1}_i$.} \label{Figure_decoup}
\end{figure}

Now, making a Girsanov change of measure similar to the one in Section
\ref{section_proof_of_uniqueness_integral}, we can decouple the particles $X^k_i$, $X^k_{i+1}$,
$X^{k-1}_i$ from the rest of the particle system, thus reducing their dynamics on the time
interval $[\sigma,\sigma']$ to the two-level dynamics:
\begin{eqnarray*}
&& X^k_i(t\wedge\sigma')-X^k_i(t\wedge\sigma) = \int_{t\wedge\sigma}^{t\wedge\sigma'}
\biggl(\mathrm{d}\tilde{W}^k_i+\frac{(1-\theta)\,\mathrm{d}s}{X^k_i(s)-X^k_{i+1}(s)}-\frac{(1-\theta)\,\mathrm{d}s}{X^k_i(s)-X^{k-1}_i(s)}\biggr), \\
&& X^k_{i+1}(t\wedge\sigma')-X^k_{i+1}(t\wedge\sigma) = \int_{t\wedge\sigma}^{t\wedge\sigma'}
\biggl(\mathrm{d}\tilde{W}^k_{i+1}+\frac{(1-\theta)\,\mathrm{d}s}{X^k_{i+1}(s)-X^k_i(s)}
-\frac{(1-\theta)\,\mathrm{d}s}{X^k_{i+1}(s)-X^{k-1}_i(s)}\biggr),\\
&& X^{k-1}_i(t\wedge\sigma')-X^{k-1}_i(t\wedge\sigma) = \int_{t\wedge\sigma}^{t\wedge\sigma'}
\mathrm{d}\tilde{W}^{k-1}_i
\end{eqnarray*}
with $\tilde{W}^k_i$, $\tilde{W}^k_{i+1}$, $\tilde{W}^{k-1}_i$ being standard Brownian motions
under the new probability measure. Next, we note that the process $X^k_{i+1}-X^k_i$ hits zero if
and only if the process \eq Z=\frac{1}{2}\Big((X_i^{k-1}-X^k_i)^2+(X_{i+1}^k-X_i^{k-1})^2\Big) \en
hits zero and in this case both events occur at the same time. Moreover, applying It\^o's formula
(see e.g.\ \cite[Section 3.3.A]{KS}) and simplifying the result, we obtain
\begin{eqnarray*}
Z(t\wedge\sigma')-Z(t\wedge\sigma) = \int_{t\wedge\sigma}^{t\wedge\sigma'} \Big((1+\theta)\,\mathrm{d}s
+(X_i^{k-1}-X^k_i)\,\mathrm{d}(\tilde{W}^{k-1}_i-\tilde{W}^k_i)
\qquad\qquad\qquad\qquad\qquad\qquad\\
+(X_{i+1}^k-X_i^{k-1})\,\mathrm{d}(\tilde{W}^k_{i+1}-\tilde{W}^{k-1}_i)\Big)
=:\int_{t\wedge\sigma}^{t\wedge\sigma'} \big((1+\theta)\,\mathrm{d}s+\mathrm{d}M\big)\quad\quad\quad\;\;
\end{eqnarray*}
where $M$ is a local martingale whose quadratic variation process satisfies
\[
\begin{split}
\langle M\rangle(t\wedge\sigma')-\langle M\rangle(t\wedge\sigma)
=\int_{t\wedge\sigma}^{t\wedge\sigma'} \Big(2\,(X_i^{k-1}-X^k_i)^2+2\,(X_{i+1}^k-X_i^{k-1})^2
\qquad\qquad\\
-(X_i^{k-1}-X^k_i)(X_{i+1}^k-X_i^{k-1})\Big)\,\mathrm{d}s,\quad t\ge0.
\end{split}
\]
We can now define the (random) time change
\[
s(t)=\inf\Big\{s\ge0:\;\;\int_{s\wedge\sigma}^{s\wedge\sigma'} (1+\theta)\,\mathrm{d}u=t\Big\},\quad 0\le t\le\int_{\sigma}^{\sigma'} (1+\theta)\,\mathrm{d}u
\]
and rewrite the stochastic integral equation for $Z$ as
\[
Z(s(t_2))-Z(s(t_1))=(t_2-t_1)+M(s(t_2))-M(s(t_1)),\quad 0\le t_1\le t_2\le\int_{\sigma}^{\sigma'} (1+\theta)\,\mathrm{d}u.
\]
A standard application of the Optional Sampling Theorem (see e.g. the proof of \cite[Theorem 4.6,
Chapter 3]{KS} for a similar argument) shows that the process $M(s(t))$, $t\ge0$ is a local
martingale in its natural filtration. In addition,
\begin{eqnarray*}
\langle M\rangle(s(t_2))-\langle M\rangle(s(t_1))
&\le&\int_{t_1}^{t_2} \frac{2\,(X_i^{k-1}(s(t))-X^k_i(s(t)))^2+2\,(X_{i+1}^k(s(t))-X_i^{k-1}(s(t)))^2}{1+\theta}\,\mathrm{d}t \\
&\le&2\,\int_{t_1}^{t_2} Z(s(t))\,\mathrm{d}t,
\quad\quad\quad 0\le t_1\le t_2\le\int_{\sigma}^{\sigma'} (1+\theta)\,\mathrm{d}u.
\end{eqnarray*}
It follows that the process
\[
\tilde{Z}(t)=
\begin{cases}
Z(s(t)) & \text{if}\;\;\;t\in\big[0,\int_{\sigma}^{\sigma'} (1+\theta)\,\mathrm{d}u\big] \\
Z\big(s\big(\int_{\sigma}^{\sigma'} (1+\theta)\,\mathrm{d}u\big)\big)+\big(t-\int_{\sigma}^{\sigma'} (1+\theta)\,\mathrm{d}u\big) & \text{if}\;\;\;t\in\big(\int_{\sigma}^{\sigma'} (1+\theta)\,\mathrm{d}u,\infty\big)
\end{cases}
\]
falls into the framework of Lemma \ref{hit0lemma}. Consequently, the original process $Z(t)$, $t\ge0$ does not reach zero on the time interval $[\sigma,\sigma']$ with probability one. Using Girsanov's Theorem again (now to go back to the original probability measure) we conclude that $X^k_{i+1}-X^k_i$ does not hit zero on $[\sigma,\sigma']$ under the original probability measure, which is the desired contradiction.
\end{proof}

Next, we study the events $C-F$ in Figure \ref{Figure_6events}. All of them can be dealt in exactly
the same manner  (in particular, using a Lyapunov function of the same form) and we  will only show
the following:

\begin{lemma}
\label{lemma_CDEF} The event
 \eq\label{event2} X^k_i(t)=X^{k-1}_i(t)=X^{k-2}_i(t) \en cannot occur in finite time.
\end{lemma}
\begin{proof}
To show the non-occurrence of the event in \eqref{event2}, we again argue by induction over $k$
and by contradiction. Assuming that the event in \eqref{event2} occurs in finite time, we may
invoke the induction hypothesis and Lemma \ref{lemma_AB_follows_CDEF} to find a random time
interval $[\sigma,\sigma']$ with $\sigma$, $\sigma'$ being stopping times and a real number
$\kappa>0$ such that the event in \eqref{event2} occurs for the first time at $\sigma'$, and
either $X^k_i(\sigma')=X^k_{i+1}(\sigma')=X^{k-1}_i(\sigma')=X^{k-2}_i(\sigma')$ and the spacings
$X^k_i-X^{k-1}_{i-1}$, $X^{k-1}_i-X^{k-2}_{i-1}$, $X^{k-2}_i-X^{k-3}_{i-1}$,
$X^{k-3}_i-X^{k-2}_i$, $X^{k-1}_{i+1}-X^{k-2}_i$, $X^{k-1}_{i+1}-X^k_{i+1}$ are bounded below by
$\kappa$ on $[\sigma,\sigma']$; or $X^k_i(\sigma')=X^{k-1}_i(\sigma')=X^{k-2}_i(\sigma')$ and the
spacings $X^k_i-X^{k-1}_{i-1}$, $X^{k-1}_i-X^{k-2}_{i-1}$, $X^{k-2}_i-X^{k-3}_{i-1}$,
$X^{k-3}_i-X^{k-2}_i$, $X^{k-1}_{i+1}-X^{k-2}_i$, $X^k_{i+1}-X^{k-1}_i$ are bounded below by
$\kappa$ on $[\sigma,\sigma']$. In the first case, we can make a Girsanov change of measure such
that under the new measure the evolution of the particles $X^k_i$, $X^k_{i+1}$, $X^{k-1}_i$,
$X^{k-2}_i$ decouples from the rest of the particle system on the time interval
$[\sigma,\sigma']$. Similarly, in the second case, we can apply a Girsanov change of measure such
that under the new measure the dynamics of the particles $X^k_i$, $X^{k-1}_i$, $X^{k-2}_i$
decouples from the dynamics of the rest of the particle configuration. Figure
\ref{Figure_decoupling4} shows schematic illustrations of these two cases.

\begin{figure}[h]
\includegraphics[width=7.5cm]{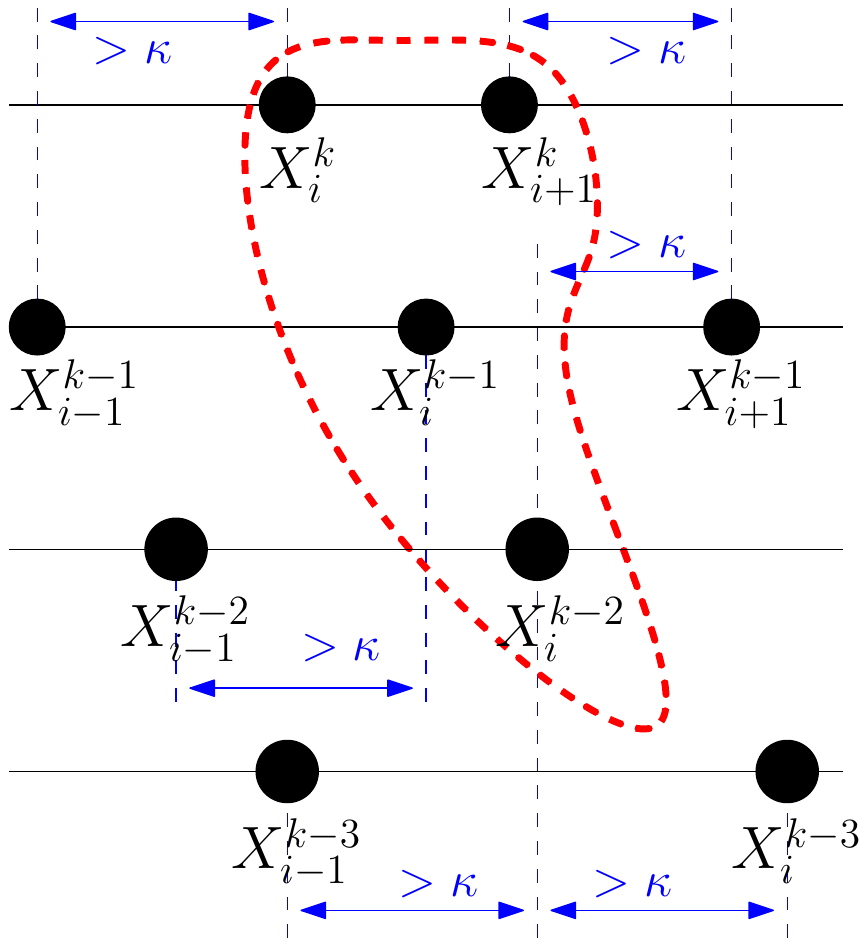}
\hfill
\includegraphics[width=7.5cm]{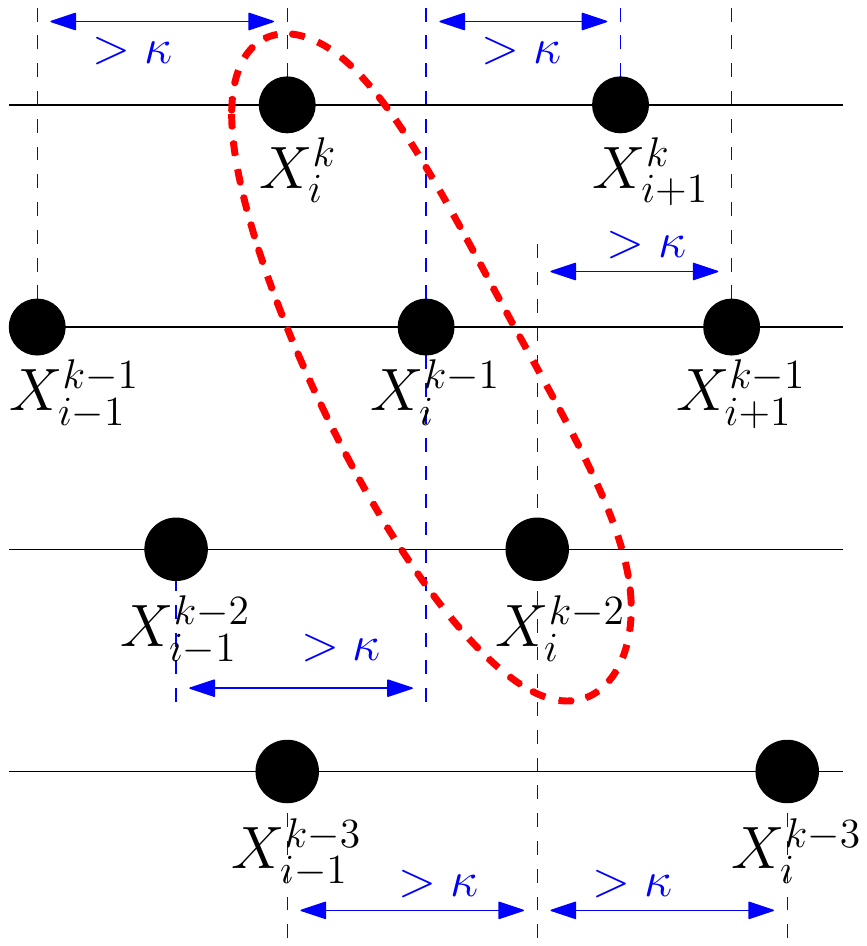}
\caption{Decoupling of the particles $X^k_i$, $X^k_{i+1}$, $X^{k-1}_i$, $X^{k-2}_i$ (left panel)
and the particles $X^k_i$, $X^{k-1}_i$, $X^{k-2}_i$ (right panel).} \label{Figure_decoupling4}
\end{figure}

We only treat the first of the two cases in detail (the second case can be dealt with by
proceeding as below with $Z:=\frac{1}{2}(R^2+S^2+2\,R\,S)$ where $R$ and $S$ are defined below).
In the first case, the decoupled particles satisfy under the new measure:
\begin{eqnarray*}
&& X^k_i(t\wedge\sigma')-X^k_i(t\wedge\sigma) =\int_{t\wedge\sigma}^{t\wedge\sigma'}
\Big(\mathrm{d}\tilde{W}^k_i+\frac{(1-\theta)\,\mathrm{d}s}{X^k_i(s)-X^k_{i+1}(s)}-\frac{(1-\theta)\,\mathrm{d}s}{X^k_i(s)-X^{k-1}_i(s)}\Big), \\
&& X^k_{i+1}(t\wedge\sigma')-X^k_{i+1}(t\wedge\sigma) =\int_{t\wedge\sigma}^{t\wedge\sigma'}
\Big(\mathrm{d}\tilde{W}^k_{i+1}+\frac{(1-\theta)\,\mathrm{d}s}{X^k_{i+1}(s)-X^k_i(s)}
-\frac{(1-\theta)\,\mathrm{d}s}{X^k_{i+1}(s)-X^{k-1}_i(s)}\Big), \\
&& X^{k-1}_i(t\wedge\sigma')-X^{k-1}_i(t\wedge\sigma) =\int_{t\wedge\sigma}^{t\wedge\sigma'}
\Big(\mathrm{d}\tilde{W}^{k-1}_i-\frac{(1-\theta)\,\mathrm{d}s}{X^{k-1}_i(s)-X^{k-2}_i(s)}\Big),\\
&&
X^{k-2}_i(t\wedge\sigma')-X^{k-2}_i(t\wedge\sigma)=\int_{t\wedge\sigma}^{t\wedge\sigma'}\,\mathrm{d}\tilde{W}^{k-2}_i
\end{eqnarray*}
where $\tilde{W}^k_i,\,\tilde{W}^k_{i+1},\,\tilde{W}^{k-1}_i,\,\tilde{W}^{k-2}_i$ are independent
standard Brownian motions under the new measure. Next, we set $R:=X^{k-1}_i-X^k_i$,
$S:=X^{k-2}_i-X^{k-1}_i$, $U:=X^k_{i+1}-X^{k-1}_i$, $B_1:=\tilde{W}^{k-1}_i-\tilde{W}^k_i$,
$B_2:=\tilde{W}^{k-2}_i-\tilde{W}^{k-1}_i$, $B_3:=\tilde{W}^k_{i+1}-\tilde{W}^{k-1}_i$ and define
$Z:=\frac{1}{2}(R^2+S^2+U^2+2\,R\,U)$. Applying It\^o's formula (see e.g.\ \cite[Section
3.3.A]{KS}) and simplifying, we obtain
\begin{eqnarray*}
Z(t\wedge\sigma')-Z(t\wedge\sigma)
=\int_{t\wedge\sigma}^{t\wedge\sigma'} \left(4+\theta+(\theta-1)\left(\frac{U(s)}{R(s)}+\frac{R(s)}{U(s)}\right)\mathrm{d}s\right)+(R+U)\mathrm{d}(B_1+B_3)+S\mathrm{d}B_2 \\
=:\int_{t\wedge\sigma}^{t\wedge\sigma'} \Big(D(s)\,\mathrm{d}s+\mathrm{d}M\Big)
\end{eqnarray*}
where $M$ is a local martingale whose quadratic variation process satisfies
\begin{eqnarray*}
\langle M\rangle(t\wedge\sigma')-\langle M\rangle(t\wedge\sigma)
=\int_{t\wedge\sigma}^{t\wedge\sigma'} 2\,(R^2+S^2+U^2+2\,R\,U)\,\mathrm{d}s
=\int_{t\wedge\sigma}^{t\wedge\sigma'} 4\,Z\,\mathrm{d}s,\quad t\ge0.
\end{eqnarray*}
Next, we introduce the time change
\[
s(t)=\inf\Big\{s\ge0:\;\int_{t\wedge\sigma}^{t\wedge\sigma'} D(u)\,\mathrm{d}u=t\Big\},\quad 0\leq t\leq\int_\sigma^{\sigma'} D(u)\,\mathrm{d}u.
\]
and rewrite the latter stochastic integral equation for $Z$ as
\[
Z(s(t_2))-Z(s(t_1))=(t_2-t_1)+M(s(t_2))-M(s(t_1)),\quad 0\le t_1\le t_2\le \int_\sigma^{\sigma'} D(u)\,\mathrm{d}u.
\]
At this point, a routine application of the Optional Sampling Theorem (see e.g. the proof of
\cite[Theorem 4.6, Chapter 3]{KS} for an argument of this type) shows that $M(s(t))$, $t\ge0$ is a
local martingale in its natural filtration. In addition,
\begin{eqnarray*}
\langle M\rangle(s(t_2))-\langle M\rangle(s(t_1))=\int_{t_1}^{t_2} \frac{4\,Z(s(t))}{D(s(t))}\,\mathrm{d}t
\le 2\int_{t_1}^{t_2} Z(s(t))\,\mathrm{d}t,\quad 0\le t_1\le t_2\le \int_\sigma^{\sigma'} D(u)\,\mathrm{d}u.
\end{eqnarray*}
Hence, the process
\[
\tilde{Z}(t)=
\begin{cases}
Z(s(t)) & \text{if}\;\;\;t\in\big[0,\int_\sigma^{\sigma'} D(u)\,\mathrm{d}u\big] \\
Z\big(s\big(\int_\sigma^{\sigma'} D(u)\,\mathrm{d}u\big)\big)+\big(t-\int_\sigma^{\sigma'} D(u)\,\mathrm{d}u\big) & \text{if}\;\;\;t\in\big(\int_\sigma^{\sigma'} D(u)\,\mathrm{d}u,\infty\big)
\end{cases}
\]
falls into the setting of Lemma \ref{hit0lemma}. The result of that lemma implies that the original process $Z(t)$, $t\ge0$ does not hit zero on the time interval $[\sigma,\sigma']$ with probability one. Changing the measure back to the original probability measure by a suitable application of Girsanov's Theorem we conclude that the same is true under the original probability measure. This is the desired contradiction.
\end{proof}

Putting together Lemmas \ref{lemma_AB_follows_CDEF} and \ref{lemma_CDEF} we deduce that $\tau_0[X]=\infty$ for all $\theta>1$.

\medskip

Finally, for $\theta\ge 2$ and any $T>0$, we have shown that the law of our process up to time
$\tau_\delta[X]\wedge T$ is absolutely continuous with respect to the law of a process comprised
of a number of Brownian motions and Bessel processes of dimension $\theta$. The definition of the
latter implies that two of its components can collide only if the corresponding Bessel process
hits zero. However, it is well-known that the Bessel process of dimension $\theta\ge 2$ does not
reach zero with probability one (see e.g. \cite[Chapter XI, Section 1]{RY}).  It follows that
$\widehat \tau_\delta[X]\ge\lim_{T\to\infty} \tau_\delta[X]\wedge T=\tau_\delta[X]$. Passing to
the limit $\delta\downarrow 0$, we conclude that $\widehat\tau_0[X]=\infty$.

\section{Convergence to multilevel Dyson Brownian Motion} \label{sec:multi}

In this section we study the diffusive scaling limit of the multilevel process $X^{multi}_{disc}$ of Definition \ref{Def_X_multi}. We start by formulating our main results.

\medskip

We fix $\theta>0$, let $\eps>0$ be a small parameter and define the $\frac{N(N+1)}{2}$-dimensional
stochastic process $Y^{mu}_\eps=((Y^{mu}_\eps)_i^k:\,1\le i\le k\le N)$ by
$$
(Y^{mu}_\eps)^k_i(t)= \eps^{1/2} \Big((X^{multi}_{disc})^k_{N+1-i}\Big(\frac{t}{\theta\eps}\Big)-\frac{t}{\eps}\Big),\;\;t\ge0:\quad 1\leq i\leq k\leq N
$$
where $(X^{multi}_{disc})_i^k$, $1\leq i\leq k\leq N$ are the coordinate processes of
$X^{multi}_{disc}$. Here, in contrast to Definition \ref{Def_X_multi}, we allow $X^{multi}_{disc}$
(or, equivalently, $Y^{mu}_\eps$) to start from an arbitrary initial condition, in particular, one
that depends on $\eps$. In addition, we use the notation
$D^{N(N+1)/2}=D([0,\infty),\rr^{N(N+1)/2})$ for the space of right-continuous paths with left
limits taking values in $\rr^{N(N+1)/2}$, endowed with the Skorokhod topology.

\begin{theorem}\label{Theorem_multilevel_tight}
Let $\theta>0$ and suppose that the family of initial conditions $Y^{mu}_\eps(0)$, $\eps\in(0,1)$
is tight on $\rr^{N(N+1)/2}$. Then the family $Y^{mu}_\eps$, $\eps\in(0,1)$ is tight on
$D^{N(N+1)/2}$.
\end{theorem}

We defer the proof of Theorem \ref{Theorem_multilevel_tight} to Section \ref{Section_multi_tight}.

\medskip

Next, we let $Y^{mu}$ be an arbitrary limit point as $\eps\downarrow0$ of the tight family
$Y^{mu}_\eps$, $\eps\in(0,1)$. For $\theta\ge 2$, we can uniquely identify the limit point with
the solution of \eqref{eq_intDBM_SDE}, thus, obtaining the diffusive scaling limit. For
$\theta\in\bigl[\frac{1}{2},2\bigr)$, we give a partial result towards such an identification.

\begin{theorem}
\label{Theorem_Limit_SDE} Let $\theta\ge 2$ (that is, $\beta=2\theta\ge 4$) and suppose that the
initial conditions $Y^{mu}_\eps(0)$, $\eps\in(0,1)$ converge as $\eps\downarrow 0$ in distribution
to a limit $Y^{mu}(0)$ which takes values in the interior of the Gelfand-Tsetlin cone
$\overline{\mathcal{G}^N}$ with probability one. Then the family $Y^{mu}_\eps$, $\eps\in(0,1)$
converges as $\eps\downarrow 0$ in distribution in $D^{N(N+1)/2}$ to the unique solution of the
system of SDEs
$$
\mathrm{d}(Y^{mu})^k_i = \Biggl(\sum_{m\neq i} \frac{(1-\theta)}{(Y^{mu})^k_i-(Y^{mu})^k_m}
-\sum_{m=1}^{k-1} \frac{(1-\theta)}{(Y^{mu})^k_i-(Y^{mu})^{k-1}_m}\Biggr)\,\mathrm{d}t +
\mathrm{d}W_i^k , \;\;1\leq i\leq k\leq N
$$
started from $Y^{mu}(0)$ and where $W_i^k$, $1\leq i\leq k\leq N$ are independent standard Brownian motions.
\end{theorem}

We give the proof of Theorem \ref{Theorem_Limit_SDE} in Section \ref{Section_multi_SDE}. We expect
Theorem \ref{Theorem_Limit_SDE} to be valid for all $\theta>1$, but we are not able to prove this
generalization.

\begin{theorem} \label{Theorem_restriction_DBM}
Suppose that the initial conditions $Y^{mu}_\eps(0)$, $\eps\in(0,1)$ converge as $\eps\downarrow 0$ in distribution to a limit $Y^{mu}(0)$ which takes values in the interior of the Gelfand-Tsetlin cone $\overline{\mathcal{G}^N}$ with probability one. In addition, suppose that the distribution of $Y^{mu}(0)$ is $\theta$-Gibbs in the sense of Definition \ref{thetaGibbsdef}.
\begin{enumerate}[(a)]
\item If $\theta\ge\frac{1}{2}$ (that is, $\beta=2\theta\geq1$), then the restriction of $Y^{mu}$ to level $N$, that is the process $((Y^{mu})^N_1,\dots,(Y^{mu})^N_N)$, is a $(2\theta)$-Dyson Brownian Motion:
$$
\mathrm{d}(Y^{mu})^N_i(t) = \sum_{m\neq i} \frac{\theta}{(Y^{mu})^N_i(t)-(Y^{mu})^N_m(t)}\,\mathrm{d}t
+ \mathrm{d}W_i ,\;\;1\leq i \leq N
$$
with $W_i^k$, $1\leq i\leq k\leq N$ being independent standard Brownian motions.
\item For any $\theta>0$ and any fixed $t>0$, the distribution of $Y^{mu}(t)$ is $\theta$-Gibbs.
\end{enumerate}
\end{theorem}

We expect the first part of Theorem \ref{Theorem_restriction_DBM} to be valid for all $\theta>0$, but we are currently not able to prove this.

\begin{proof}[Proof of Theorem \ref{Theorem_restriction_DBM}]
The theorem follows from a combination of Proposition \ref{prop_intertwining_disc}, Proposition \ref{proposition_convergence_to_theta_Gibbs} and Theorem \ref{theorem_DBM}.
\end{proof}

\begin{corollary} \label{cor_zero_initial} Take any $\theta>0$ and suppose that $Y^{mu}_\eps(0)=0\in\rr^{N(N+1)/2}$, $\eps\in(0,1)$. Then, for any $t\ge 0$, the distribution of $Y^{mu}(t)$ is given by the Hermite $\beta=2\theta$ corners process of variance $t$ (see Definition \ref{defbetacorner}).
\end{corollary}
\begin{rmk}
Since for any $t>0$, the Hermite $\beta=2\theta$ corners process of variance $t$ is supported by the interior of the Gelfand-Tsetlin cone $\overline{\mathcal{G}^N}$, it follows that Theorem \ref{Theorem_Limit_SDE} (for $\theta\ge
2$) can be applied in this case as well. Consequently, for $\theta\ge2$, the process $Y^{mu}$ started from the zero initial condition is a diffusion that combines Dyson Brownian motions and corners processes into a single picture as desired.
\end{rmk}

\begin{proof}[Proof of Corollary \ref{cor_zero_initial}]
The corollary is a consequence of Proposition \ref{prop_intertwining_disc} and Corollary \ref{corollary_convergence_multi_level_fixed_time}.
\end{proof}

%

The rest of this section is devoted to the proof of Theorems \ref{Theorem_multilevel_tight} and \ref{Theorem_Limit_SDE}. The structure of the proof is similar to that of Theorem \ref{theorem_DBM}: In Section \ref{Section_multi_rates} we analyze the asymptotic behavior of the jump rates of the processes $Y^{mu}_\eps$, $\eps\in(0,1)$, in Section \ref{Section_multi_rates} we use this asymptotics to prove that the family $Y^{mu}_\eps$, $\eps\in(0,1)$ is tight, and in Section \ref{Section_multi_SDE} we deduce the SDE \eqref{eq_intDBM_SDE} for subsequential limits as $\eps\downarrow0$ of this family when $\theta\ge2$. We omit the details in the parts that are parallel to the arguments of Section \ref{Section_DBM}.

\subsection{Step 1: Rates}\label{Section_multi_rates}
We start by noting that, for each $\eps\in(0,1)$, $Y^{mu}_\eps$ is a continuous time Markov process with state space $\GG$, a (constant) drift of $-\eps^{-1/2}$ in each coordinate and jump rates
$$
q^{mu}_\eps(y,y',t)=\frac{1}{\theta\eps}\,q_{\frac{t}{\theta}\eps^{-1}+\hat{y}\eps^{-1/2}\,\rightarrow\,\frac{t}{\theta}\eps^{-1}+\hat{y'}\eps^{-1/2}}
$$
where $\hat{y}$, $\hat{y'}$ are the vectors obtained from $y$, $y'$ by reordering the coordinates on each level in decreasing order, and intensities $q$ are given by \eqref{eq_jump_intensity}. Write $y'\approx_\eps y$ for vectors $y,y'\in\GG$ such that $y'$ can be obtained from $y$ by increasing one coordinate (say, $y^k_i$) by $\eps^{1/2}$, and, if necessary, by increasing other coordinates as well to preserve the interlacing condition (in the sense of the push interaction as explained after \eqref{eq_jump_intensity}). Clearly, $q^{mu}_\eps(y,y',t)=0$ unless $y'\approx_\eps y$. As we will see, in fact, $q^{mu}_\eps(y,y',t)$ does not depend on $t$.

\begin{lemma} \label{Lemma_expansion_of_rates_multi}
For any sequence of vectors $y'\approx_\eps y$ such that $y'$ differs from $y$ in exactly one coordinate and any fixed $k\in\{1,\dots,N\}$, $i\in\{1,2,\ldots,k\}$ one has the following $\eps\downarrow0$ asymptotics:
\begin{equation}
\label{eq_multi_asymp} q^{mu}_\eps(y,y',t)=\eps^{-1}+\eps^{-1/2}\Big(\sum_{m\neq i}
\frac{1-\theta}{y_i^k-y^k_m}-\sum_{m=1}^{k-1} \frac{1-\theta}{y_i^k-y_m^{k-1}}\Big)+O(1)
\end{equation}
with a uniform $O(1)$ remainder on compact subsets of the interior of the Gelfand-Tsetlin cone $\overline{\GGo}$.
\end{lemma}
\begin{proof}
 As above, we write $\hat y_i^k$ for $y_{k+1-i}^k$ and $(\hat y')_i^k$ for $(y')_{k+1-i}^k$. Using \eqref{eq_jump_intensity} and arguing as
 in Lemma \ref{Lemma_expansion_of_rates_single}, we write $q^{mu}_\eps(y,y',t)$ as
\begin{eqnarray*}
\eps^{-1}\prod_{l=1}^{i-1} \frac{(\hat y_l^k-\hat y_i^k)\eps^{-1/2}-1+\theta(i-l+1)}{(\hat
y_l^k-\hat y_i^k)\eps^{-1/2}-1+\theta(i-l)}\cdot
\frac{(\hat y_l^k-\hat y_i^k)\eps^{-1/2}+\theta(i-l-1)}{(\hat y_l^k-\hat y_i^k)\eps^{-1/2}+\theta(i-l)} \\
\prod_{1\leq m\leq n\leq k-1} \frac{\big(((\hat y')^{k-1}_m-(\hat
y')^{k-1}_n)\eps^{-1/2}+\theta(n-m)+\theta\big)_{((\hat y')^{k-1}_n-(\hat y')^k_{n+1})\eps^{-1/2}}}
{\big(((\hat y')^{k-1}_m-(\hat y')^{k-1}_n)\eps^{-1/2}+\theta(n-m)+1\big)_{((\hat y')^{k-1}_n-(\hat y')^k_{n+1})\eps^{-1/2}}}\\
\frac{\big(((\hat y')^k_m-(\hat y')^{k-1}_n)\eps^{-1/2}+\theta(n-m)+1\big)_{((\hat
y')^{k-1}_n-(\hat y')^k_{n+1})\eps^{-1/2}}}
{\big(((\hat y')^k_m-(\hat y')^{k-1}_n)\eps^{-1/2}+\theta(n-m)+\theta\big)_{((\hat y')^{k-1}_n-(\hat y')^k_{n+1})\eps^{-1/2}}}\\
\prod_{1\leq m\leq n\leq k-1}
\frac{\big((y^{k-1}_m-y^{k-1}_n)\eps^{-1/2}+\theta(n-m)+1\big)_{(y^{k-1}_n-y^k_{n+1})\eps^{-1/2}}}
{\big((y^{k-1}_m-y^{k-1}_n)\eps^{-1/2}+\theta(n-m)+\theta\big)_{(y^{k-1}_n-y^k_{n+1})\eps^{-1/2}}}\\
\frac{\big((y^k_m-y^{k-1}_n)\eps^{-1/2}+\theta(n-m)+\theta\big)_{(y^{k-1}_n-y^k_{n+1})\eps^{-1/2}}}
{\big((y^k_m-y^{k-1}_n)\eps^{-1/2}+\theta(n-m)+1\big)_{(y^{k-1}_n-y^k_{n+1})\eps^{-1/2}}}.
\end{eqnarray*}
Using the fact that $y$ and $y'$ differ only in one coordinate, we can simplify the latter
expression to
\begin{eqnarray*}
\eps^{-1}\prod_{m=1}^{i-1} \frac{(\hat y_m^k-\hat y_i^k)\eps^{-1/2}-1+\theta(i-m+1)}{(\hat
y_m^k-\hat y_i^k)\eps^{-1/2}-1+\theta(i-m)}\cdot
\frac{(\hat y_m^k-\hat y_i^k)\eps^{-1/2}+\theta(i-m-1)}{(\hat y_m^k-\hat y_i^k)\eps^{-1/2}+\theta(i-m)} \\
\prod_{m=1}^{i-1} \frac{(\hat y_m^{k-1}-\hat y_i^k)\eps^{-1/2}+\theta(i-1-m)}{(\hat y_m^{k-1}-\hat
y_i^k)\eps^{-1/2}-1+\theta(i-m)}
\prod_{n=i}^{k-1} \frac{(\hat y_i^k-\hat y_{n+1}^k)\eps^{-1/2}+\theta(n-i)+1}{(\hat y_i^k-\hat y_n^{k-1})\eps^{-1/2}+\theta(n-i)+1}\\
\prod_{m=1}^{i-1} \frac{(\hat y_m^k-\hat y_i^k)\eps^{-1/2}+\theta(i-m)-1}{(\hat y_m^k-\hat
y_i^k)\eps^{-1/2}+\theta(i-1-m)}
\prod_{n=i}^{k-1} \frac{(\hat y_i^k-\hat y_n^{k-1})\eps^{-1/2}+\theta(n-i+1)}{(\hat y_i^k-\hat y_{n+1}^k)\eps^{-1/2}+\theta(n-i+1)}\\
=\eps^{-1}\prod_{m=1}^{i-1} \frac{(\hat y_m^k-\hat y_i^k)\eps^{-1/2}-1+\theta(i-m+1)}{(\hat
y_m^k-\hat y_i^k)\eps^{-1/2}+\theta(i-m)} \cdot
\frac{(\hat y_m^{k-1}-\hat y_i^k)\eps^{-1/2}+\theta(i-1-m)}{(\hat y_m^{k-1}-\hat y_i^k)\eps^{-1/2}-1+\theta(i-m)}\\
\prod_{n=i}^{k-1} \frac{(\hat y_i^k-\hat y_{n+1}^k)\eps^{-1/2}+\theta(n-i)+1}{(\hat y_i^k-\hat
y_{n+1}^k)\eps^{-1/2}+\theta(n-i+1)} \cdot \frac{(\hat y_i^k-\hat
y_n^{k-1})\eps^{-1/2}+\theta(n-i+1)}{(\hat y_i^k-\hat y_n^{k-1})\eps^{-1/2}+\theta(n-i)+1}.
\end{eqnarray*}
Expanding the last expression into a power series in terms of $\eps^{-1/2}$ we get
$$
 \eps^{-1} +\eps^{-1/2}\left(\sum_{m=1}^{i-1} \left(\frac{\theta-1}{\hat y_m^k-\hat y_i^k}+ \frac{1-\theta}{\hat y_m^{k-1}-\hat y_i^k}\right)\right)+ \left(\sum_{n=i}^{k-1}
 \left(\frac{1-\theta}{\hat y_i^k-\hat y_{n+1}^k}+\frac{\theta-1}{\hat y_i^k-\hat y_n^{k-1}}\right)
 \right) +O(1).
$$
The lemma now readily follows.
\end{proof}

\subsection{Step 2: Tightness}\label{Section_multi_tight}
We show next that the family $Y^{mu}_\eps$, $\eps\in(0,1)$ is tight on $D^{N(N+1)/2}$. To this end, we aim to apply the necessary and sufficient conditions for tightness of \cite[Corollary 3.7.4]{EK} which amount to showing that, for any fixed $t\geq0$, the random variables $Y^{mu}_\eps(t)$, $\eps\in(0,1)$ are tight on $\rr^{N(N+1)/2}$ and that, for every $\Delta>0$ and $T>0$, there exists a $\delta>0$ such that
\[
\limsup_{\eps\downarrow0}\,\pp\Big(\sup_{0\leq s<t\leq T,t-s<\delta}
\big|(Y^{mu}_\eps)^k_i(t)-(Y^{mu}_\eps)^k_i(s)\big|>\Delta\Big)<\Delta,\quad 1\leq i\leq k\leq N.
\]
We start by explaining how to deal with $(Y^{mu}_\eps(t))_+$ (the vector of positive parts of components of $Y^{mu}_\eps(t)$) and
\eq\label{tightmulti}
\sup_{0\leq s<t\leq T,t-s<\delta}
\big((Y^{mu}_\eps)^k_i(t)-(Y^{mu}_\eps)^k_i(s)\big),\quad 1\leq i\leq k\leq N.
\en

\medskip

We argue by induction over the value of $k$. For $k=1$, it is sufficient to observe that $(Y^{mu}_\eps)^1_1$ is a Poisson process with jump size $\eps^{1/2}$, jump rate $\eps^{-1}$ and drift $-\eps^{-1/2}$ and, hence, converges to a standard Brownian motion in the limit $\eps\downarrow0$. Therefore, the necessary and sufficient conditions of \cite[Corollary 3.7.4]{EK} hold for $(Y^{mu}_\eps)^1_1$.

\medskip

We now fix some $k\geq2$ and distinguish the cases $0<\theta\leq 1$ and $\theta>1$. In the first
case, we consider first $i=k$.  It is easy to see from the formulas for the jump rates in the
proof of Lemma \ref{Lemma_expansion_of_rates_multi} that for $0<\theta<1$, whenever the spacing
$(Y^{mu}_\eps)^k_k-(Y^{mu}_\eps)^{k-1}_{k-1}$ exceeds $\Delta/3$, the jump rate to the right of
$(Y^{mu}_\eps)^k_k$ is bounded above by $\eps^{-1}$. For $\theta=1$ the jump rate is \emph{equal}
to $\eps^{-1}$. Hence, arguing as in Section \ref{Section_tightness_single}, we conclude that for
$0<\theta\le 1$, $(Y^{mu}_\eps)^k_k$ can be coupled with a Poisson process with jump size
$\eps^{1/2}$, jump rate $\eps^{-1}$ and drift $-\eps^{-1/2}$ in such a way that, whenever
$(Y^{mu}_\eps)^k_k-(Y^{mu}_\eps)^{k-1}_{k-1}$ exceeds $\Delta/3$ and $(Y^{mu}_\eps)^k_k$ has a
jump to the right, the Poisson process jumps to the right as well. Therefore, the convergence of
such Poisson processes to a standard Brownian motion and the necessary and sufficient conditions
of \cite[Corollary 3.7.4]{EK} for them imply the corresponding conditions for
$((Y^{mu}_\eps)^k_k(t))_+$ and the quantity in \eqref{tightmulti} with $i=k$.

\medskip

For $i\in\{1,2,\ldots,k-1\}$, the quantity $((Y^{mu}_\eps)^k_i(t))_+$ can be bounded above by the
quantity $((Y^{mu}_\eps)^{k-1}_i(t))_+$ and the latter satisfies the required condition by the
induction hypothesis. Moreover, the formula for the jump rates \eqref{eq_multi_asymp} reveals that, whenever
$(Y^{mu}_\eps)^k_i-(Y^{mu}_\eps)^{k-1}_{i-1}$ and
$(Y^{mu}_\eps)^{k-1}_i-(Y^{mu}_\eps)^k_i$ both exceed $\Delta/4$, the jump rate to the right of
$(Y^{mu}_\eps)^k_i$ is bounded above by \eq
\eps^{-1}+\frac{\eps^{-1/2}(1-\theta)}{(Y^{mu}_\eps)^{k-1}_i(t)-(Y^{mu}_\eps)^k_i(t)}+O(1) \leq
\eps^{-1}+4\,\eps^{-1/2}(1-\theta)/\Delta+O(1). \en Hence, $(Y^{mu}_\eps)^k_i$ can be coupled with
a Poisson process with jump size $\eps^{1/2}$, jump rate given by the right-hand side of the latter
inequality and drift $-\eps^{-1/2}$ in such a way that, whenever
$(Y^{mu}_\eps)^k_i-(Y^{mu}_\eps)^{k-1}_{i-1}$ and $(Y^{mu}_\eps)^{k-1}_i-(Y^{mu}_\eps)^k_i$ both
exceed $\Delta/4$ and $(Y^{mu}_\eps)^k_i$ has a jump to the right, the Poisson process jumps to the
right as well. Thus, the convergence of such Poisson processes to Brownian motion with drift
$4\,(1-\theta)/\Delta$ and the necessary and sufficient conditions of  \cite[Corollary 3.7.4]{EK}
for them imply the corresponding control on the quantities in \eqref{tightmulti}.

\medskip

In the case $\theta>1$, we first consider $i=1$. From the formulas for the jump rates in the proof of Lemma \ref{Lemma_expansion_of_rates_multi} it is not hard to see that the jump rate to the right of the process $(Y^{mu}_\eps)^k_1$ is bounded above by $\eps^{-1}$. Therefore, it can be coupled with a Poisson process with jump size $\eps^{1/2}$, jump rate $\eps^{-1}$ and drift $-\eps^{-1/2}$ in such a way that, whenever $(Y^{mu}_\eps)^k_1$ has a jump to the right, the Poisson process jumps to the right as well. Thus, the convergence of such Poisson processes to a Brownian motion and the necessary and sufficient conditions of \cite[Corollary 3.7.4]{EK} for them give the desired control on $((Y^{mu}_\eps)^k_1(t))_+$ and the quantity in \eqref{tightmulti} with $i=1$.

\medskip

For $i\in\{2,3,\ldots,k\}$, the formulas for the jump rates in the proof of Lemma \ref{Lemma_expansion_of_rates_multi} reveal that, whenever $(Y^{mu}_\eps)^k_i-(Y^{mu}_\eps)^{k-1}_{i-1}$ exceeds $\Delta/3$, the jump rate to the right of $(Y^{mu}_\eps)^k_i$ is bounded above by
\eq
\eps^{-1}-\frac{\eps^{-1/2}(1-\theta)}{(Y^{mu}_\eps)^k_i(t)-(Y^{mu}_\eps)^{k-1}_{i-1}(t)}+O(1)
\leq \eps^{-1}-3\,\eps^{-1/2}(1-\theta)/\Delta+O(1).
\en
Hence, $(Y^{mu}_\eps)^k_i$ can be coupled with a Poisson process with jump size $\eps^{1/2}$, jump rate given by the right-hand side of the last inequality and drift $-\eps^{-1/2}$, so that, whenever $(Y^{mu}_\eps)^k_i-(Y^{mu}_\eps)^{k-1}_{i-1}$ exceeds $\Delta/3$ and $(Y^{mu}_\eps)^k_i$ has a jump to the right, the Poisson process jumps to the right as well. Therefore, the convergence of such Poisson processes to Brownian motion with drift $-3\,(1-\theta)/\Delta$ and the necessary and sufficient conditions of \cite[Corollary 3.7.4]{EK} for them yield the corresponding conditions
for $((Y^{mu}_\eps)^k_i(t))_+$ and the quantities in \eqref{tightmulti}.

\medskip

Finally, we note that the quantities $(Y^{mu}_\eps(t))_-$ (the vector of negative parts of components of $Y^{mu}_\eps(t)$) and
\[
\sup_{0\leq s<t\leq T,t-s<\delta} -\big((Y^{mu}_\eps)^k_i(t)-(Y^{mu}_\eps)^k_i(s)\big),\quad 1\leq
i\leq k\leq N
\]
can be analyzed in a similar manner (however, now by moving from the leftmost to the rightmost particle on every level for $0<\theta\leq 1$ and vice versa for $\theta>1$). By combining everything together and using \cite[Corollary 3.7.4]{EK} we conclude that the family $Y^{mu}_\eps$, $\eps\in(0,1)$ is tight.

\medskip

We also note that, since the maximal size of the jumps tends to zero as $\eps\downarrow 0$, any limit point of the family $Y^{mu}_\eps$, $\eps\in(0,1)$ as $\eps\downarrow 0$ must have continuous paths (see e.g. \cite[Theorem 3.10.2]{EK}).

\subsection{Step 3: SDE for subsequential limits}
\label{Section_multi_SDE}

Writing $Y^{mu}$ for an arbitrary limit point as $\eps\downarrow0$ of the tight family $Y^{mu}_\eps$, $\eps\in(0,1)$ as
before, our goal now is to prove that $Y^{mu}$ solves the SDE \eqref{eq_intDBM_SDE}. We pick a sequence of $Y^{mu}_\eps$ which converges to $Y^{mu}$ in law, and by virtue of the Skorokhod Embedding Theorem (see e.g. \cite[Theorem 3.5.1]{Du}) may assume that all processes involved are defined on the same probability space and that the convergence holds in the almost sure sense. In the rest of this section all the limits $\eps\downarrow 0$ are taken along such a sequence.

\medskip

Define the set of test functions
\[
\mathcal{F}^{mu}:=\big\{f\in
C_0^\infty(\overline{\mathcal{G}^N})|\;\exists\,\delta>0:\;f(x)=0\;\,\mathrm{whenever}\;\,\mathrm{dist}(x,
\partial\overline{\mathcal{G}^N})\leq\delta\big\}.
\]
Hereby, $\partial\overline{\mathcal{G}^N}$ stands for the boundary of $\overline{\mathcal{G}^N}$ and $\mathrm{dist}$ denotes the usual Euclidean distance. In addition, for every test function $f\in\mathcal{F}^{mu}$, define the process
\begin{eqnarray*}
M^f(t):=f(Y^{mu}(t))-f(Y^{mu}(0))- \sum_{1\leq i\leq k\leq N} \int_0^t
\frac{1}{2}\,f_{y^k_iy^k_i}(Y^{mu}(s))\,\mathrm{d}s
\qquad\qquad\qquad\qquad\qquad\qquad\quad\\
- \sum_{1\leq i\leq k\leq N} \int_0^t \Big(\sum_{m\neq i} \frac{1-\theta}{(Y^{mu})^k_i(s)-(Y^{mu})^k_m(s)}
-\sum_{m=1}^{k-1} \frac{1-\theta}{(Y^{mu})^k_i(s)-(Y^{mu})^{k-1}_m(s)}\Big)f_{y^k_i}(Y^{mu}(s))\,\mathrm{d}s.
\end{eqnarray*}
Our first aim is to show that each $M^f$ is a martingale. To this end, we fix an $f\in\mathcal{F}^{mu}$ and
consider the family of martingales
\begin{eqnarray*}
M^f_\eps(t)&:=&f(Y^{mu}_\eps(t))-f(Y^{mu}_\eps(0))
+ \int_0^t \sum_{1\leq i\leq k\leq N} \eps^{-1/2}\,f_{y^k_i}(Y^{mu}_\eps(s))\,\mathrm{d}s \\
&& -\int_0^t\sum_{y'\approx_\eps Y^{mu}_\eps(s)} q^{mu}_\eps(Y^{mu}_\eps(s),\hat y',s)\big(f(y')-f(Y^{mu}_\eps(s))\big)
\,\mathrm{d}s,\quad \eps>0
\end{eqnarray*}
where the notations are the same as in Section \ref{Section_multi_rates}. Now, one can argue as in Section \ref{Section_SDE_DBM_limits}, Step 3a to conclude that $M^f$ is a martingale, as the $\eps\downarrow0$ limit of the martingales $M^f_\eps$ which can be bounded uniformly on every compact time interval.

\medskip

Next, for each $\delta>0$ and $K>0$, we define the stopping time
\begin{eqnarray*}
\widehat\tau_{\delta,K}:=\inf\big\{t\ge0:\;\mathrm{dist}(Y^{mu}(t),\partial\overline{\mathcal{G}^N})\le\delta
\quad\mathrm{or}\quad|(Y^{mu})^k_i(t)|\ge K\;\mathrm{for\;some\;}(k,i)\big\}.
\end{eqnarray*}
Now, note that for every function $\overline{\mathcal{G}^N}\rightarrow\rr$ of the form $x\mapsto x^k_i$ or $x\mapsto x^k_i\,x^{k'}_{i'}$, there is a function $f\in\mathcal{F}^{mu}$ which coincides with that function on
\[
\{x\in\overline{\mathcal{G}^N}:\;\mathrm{dist}(x,\partial\overline{\mathcal{G}^N})\geq\delta,\;
|x^k_i|\le K\;\mathrm{for\;all\;}(k,i)\}.
\]
Combining this observation, the Optional Sampling Theorem and the conclusion of the preceeding
paragraph we deduce that all processes of the two forms
\begin{eqnarray*}
M^{(k,i)}(t)&:=&(Y^{mu})^k_i(t\wedge\widehat\tau_{\delta,K})-(Y^{mu})^k_i(0) \\
&&-\int_0^{t\wedge\widehat\tau_{\delta,K}} \Big(\sum_{m\neq i} \frac{1-\theta}{(Y^{mu})^k_i(s)-(Y^{mu})^k_m(s)}
-\sum_{m=1}^{k-1} \frac{1-\theta}{(Y^{mu})^k_i(s)-(Y^{mu})^{k-1}_m(s)}\Big)\,\mathrm{d}s, \\
M^{(k,i),(k',i')}(t)&:=&(Y^{mu})^k_i(t\wedge\widehat\tau_{\delta,K})(Y^{mu})^{k'}_{i'}(t\wedge\widehat\tau_{\delta,K})
-(Y^{mu})^k_i(0)(Y^{mu})^{k'}_{i'}(0)-\mathbf{1}_{\{k=k',i=i'\}}\cdot t\wedge\widehat\tau_{\delta,K} \\
&& -\int_0^{t\wedge\widehat\tau_{\delta,K}} \Big(\sum_{m\neq i} \frac{(1-\theta)(Y^{mu})^{k'}_{i'}(s)}
{(Y^{mu})^k_i(s)-(Y^{mu})^k_m(s)}
-\sum_{m=1}^{k-1} \frac{(1-\theta)(Y^{mu})^{k'}_{i'}(s)}{(Y^{mu})^k_i(s)-(Y^{mu})^{k-1}_m(s)}\Big)\,\mathrm{d}s \\
&& -\int_0^{t\wedge\widehat\tau_{\delta,K}} \Big(\sum_{m\neq i'} \frac{(1-\theta)(Y^{mu})^k_i(s)}
{(Y^{mu})^{k'}_{i'}(s)-(Y^{mu})^{k'}_m(s)}
-\sum_{m=1}^{k'-1} \frac{(1-\theta)(Y^{mu})^k_i(s)}{(Y^{mu})^{k'}_{i'}(s)-(Y^{mu})^{k'-1}_m(s)}\Big)\,\mathrm{d}s
\end{eqnarray*}
are martingales.  It now remains to follow the proof of \cite[Chapter 5 ,Proposition 4.6]{KS}
(only replacing every occurence of $t$ by $t\wedge\widehat\tau_{\delta,K}$) to end up with the
system of stochastic integral equations
\begin{eqnarray*}
&&(Y^{mu})^k_i(t\wedge\widehat\tau_{\delta,K})-(Y^{mu})^k_i(0) \\
&&=\int_0^{t\wedge\widehat\tau_{\delta,K}} \Big(\sum_{m\neq i} \frac{1-\theta}{(Y^{mu})^k_i(s)-(Y^{mu})^k_m(s)}
-\sum_{m=1}^{k-1} \frac{1-\theta}{(Y^{mu})^k_i(s)-(Y^{mu})^{k-1}_m(s)}\Big)\,\mathrm{d}s
+\mathrm{W}^k_i(t\wedge\widehat\tau_{\delta,K}),\\
&&\qquad\qquad\qquad\qquad\qquad\qquad\qquad\qquad\qquad\qquad\qquad\qquad\qquad\qquad\qquad\qquad\qquad\;\;
1\leq i\leq k\leq N
\end{eqnarray*}
where $W^k_i$, $1\leq i\leq k\leq N$ are independent standard Brownian motions, possibly defined
on an extension of the underlying probability space. At this stage, one can repeat the argument at
the end of Section \ref{section_stopping_time_infinity} to show that $\lim_{K\uparrow\infty}
\widehat\tau_{\delta,K}\geq\tau_{\delta}[Y^{mu}]$ and then combine Propositions
\ref{Proposition_uniqueness_integral} and \ref{proposition_multlilevel_stopping} to end up with
the SDE of Theorem \ref{Theorem_Limit_SDE} as desired.
\begin{rmk} In Section \ref{Section_SDE_DBM_limits} we used Lemma \ref{Lemma_covariance_single}
instead of \cite[Chapter 5, Proposition 4.6]{KS}, but we could have used the latter as well. On
the other hand it is not straightforward to generalize Lemma \ref{Lemma_covariance_single} to the
setting of the current section, because there is no multilevel analogue for Proposition
\ref{Prop_sumpoisson}.
\end{rmk}

\end{document}